\documentclass[12pt]{article}

\usepackage{amsmath, amssymb, amsthm}
\usepackage{geometry}
\usepackage{tikz}
\usepackage{hyperref}
\usepackage{setspace}
\usepackage{graphicx}
\usepackage{enumerate}
\usepackage{comment}
\makeatletter
\newcommand{\myitem}[1]{%
	\item[#1]\protected@edef\@currentlabel{#1}%
}
\makeatother

\usepackage{clipboard}
\numberwithin{equation}{section}
\newtheorem{mytheorem}{Theorem}
\newtheorem{mylemma}{Lemma}
\newtheorem{defi}{Definition}[section]
\numberwithin{mytheorem}{subsection}
\numberwithin{mylemma}{subsection}
\usepackage{amsbsy}
\usepackage{graphicx}
\usepackage[T1]{fontenc}
\usepackage{lmodern}
\usepackage{imakeidx}
\usepackage{amssymb}
\usepackage{cite}
\newcommand{\rnn}{\mathbb{R}^{N}}

\newcommand{\ra} {\rightarrow}

\newcommand{\la} {\lambda}

\newcommand{\noi} {\noindent}
\newcommand{\na} {\nabla}

\newcommand{\mb} {\mathbb}

\newcommand{\I}{\int\limits_}
\newenvironment{myproof}[2] {\paragraph{Proof of {#1} {#2} :}}{\hfill$\square$}

\usepackage[english]{babel}
\usepackage[dvipsnames]{xcolor}

\newtheorem{theorem}{Theorem}[subsection]
\newtheorem{lemma}[theorem]{Lemma}
\newtheorem{proposition}[theorem]{Proposition}

\theoremstyle{definition}

\theoremstyle{remark}

\title{Ground state solutions of mixed local-nonlolcal equations with Hartree type nonlinearities}
\author{Gurdev Chand Anthal, Prashanta Garain and Nidhi Nidhi}

\date{\today}

\begin{document}

\maketitle

\begin{abstract}
We study a class of mixed local-nonlocal equations with Hartree-type nonlinearities of the form
\begin{equation}\label{meqnab}
-\Delta u + (-\Delta)^s u + u = (I_\alpha * F(u))\,F'(u) \quad \text{in } \mathbb{R}^N,
\end{equation}
where $N \geq 3$, $s \in (0,1)$, and $F \in C^1(\mathbb{R},\mathbb{R})$ satisfies Berestycki--Lions type assumptions. The equation combines the classical Laplacian with the fractional Laplacian, while the Hartree-type nonlinearity is given by a nonlocal convolution term involving the Riesz potential $I_\alpha$, with $\alpha \in (0,N)$. We prove the existence of ground state solutions to \eqref{meqnab}. To this end, we establish regularity properties and derive a Poho\v{z}aev-type identity for general weak solutions of \eqref{meqnab}. Moreover, we obtain symmetry properties of ground state solutions of \eqref{meqnab} via polarization methods.
\end{abstract}

\noi {Keywords: Mixed local-nonlocal problem, Hartree-type nonlinearity, Ground state soution, regularity, symmetry.}

\noi{\textit{2020 Mathematics Subject Classification: 35M10, 35J61, 35A01.}

\tableofcontents

\section{Introduction and main results}
\subsection{Introduction}
We study the following class of Mixed local-nonlocal equation with Choquard nonlinearity given by
	\begin{equation}\label{1.1}
		-\Delta	u+(-\Delta)^su +u =  (I_{\alpha}*F(u))F'(u)\text{ in } \mathbb{R}^N,
	\end{equation}
	where $N\geq 3$ and $\alpha \in (0,N)$. Here $\Delta$ is the classical Laplace operator and for $0<s<1$, the fractional Laplace operator $(-\Delta)^s$ is defined as:
	$$(-\Delta)^su(x)=C(N,s)\text{P.V.}\int_{\mathbb{R}^N}\frac{u(x)-u(y)}{|x-y|^{N+2s}}dy,
    $$
	where $C(N,s)$ is the normalizing constant given by 
	$$C(N,s)=\left(\int_{\mathbb{R}^N}\frac{1-\cos(x)}{|x|^{N+2s}}\right)^{-1},$$
	and P.V. denotes the principal value.\\
    Moreover, $I_{\alpha}$ is the Riesz potential of order $\alpha \in (0,N)$ defined by
	$$I_{\alpha}(x)=\frac{A_{N,\alpha}}{|x|^{N-\alpha}},\text{ with } A_{N,\alpha}=\frac{\Gamma(\frac{N-\alpha}{2})}{\pi^{\frac{N}{2}}2^{\alpha}\Gamma(\frac{\alpha}{2})}\; \text{  for every } x\in \mathbb{R}^N \setminus \{0\}.$$
	We assume that $F\in C^1(\mathbb{R};\mathbb{R})$ and $f:=F'$ satisfies the following Beretyscki-Lions type conditions:
	\begin{enumerate}
		\myitem{(f1)}\label{f1} There exists a constant $C>0$ such that 
		$|tf(t)|\leq C(|t|^{\frac{N+\alpha}{N}}+|t|^{\frac{N+\alpha}{N-2}}) \text{ for all } t\in \mathbb{R},$
		\myitem{(f2)}\label{f2} $\displaystyle \lim_{t\rightarrow 0}\frac{F(t)}{|t|^{\frac{N+\alpha}{N}}}=0=\lim_{|t|\rightarrow \infty}\frac{F(t)}{|t|^{\frac{N+\alpha}{N-2}}}$,
		\myitem{(f3)}\label{f3} there exists $t_0\in \mathbb{R}\setminus\{0\}$ such that $F(t_0)\neq 0$,
        \myitem{(f4)}\label{f4} $f$ is nondecreasing and there exists a constant $\bar{k}>0$ such that $|f(t)| \leq \bar{k}|t|$ for all $t \in \mathbb R$.
	\end{enumerate} 
    The convolution $I_\alpha*F(u)$ is defined by
\[
(I_\alpha * F(u))(x) := \int_{\mathbb{R}^N} I_\alpha(x-y)\, F(u)(y)\, dy.
\]
     The study of problems involving nonlinearities with such growth conditions can be dated back to the work of Berestycki and Lions; see \cite{Berestycki-Lions1, Berestycki-Lions2}, where they studied the following localised version:
\begin{equation}\label{B-L_equation}
		-\Delta u + u = h(u) \text{ in } \mathbb{R}^N.
\end{equation}
	Here, the nonlinearity $h$ satisfies the following assumptions:
    \begin{itemize}
        \item $$-\infty<\liminf_{s\rightarrow0^+} \frac{h(s)}{s}\leq \limsup_{s\rightarrow 0^+}\frac{h(s)}{s}=-m<0;$$
        \item $$-\infty \leq \limsup_{s\rightarrow + \infty}\frac{h(s)}{s^{2^*-1}}\leq 0;$$
        \item there exists $t_0>0$ such that $H(t_0)>0.$
    \end{itemize}
where $H(t)=\int_{0}^{t}h(\sigma)d\sigma$.
Such problems are studied when one is looking for solitary waves or stationary states of nonlinear Klein-Gordon and Schr$\ddot{\text{o}}$dinger equations. By using the direct method of minimization in combination with P\'{o}lya-Szeg\"{o} inequality for Schwartz symmetrization and a scaling argument, Berestycki and Lions showed that the solution of the constrained minimization problem
\begin{align*}
    \min\left\{ \I{\mb R^N} |\na u|^2
\,dx: u \in H^1(\mb R^N)~\text{and}~\I{\mb R^N} \Big(H(u)-\frac{|u|^2}{2}\Big)\,dx =1\right\},
\end{align*}
is infact a radial ground state solution of \eqref{B-L_equation}, see \cite[Theorem~1]{Berestycki-Lions1}. Furthermore, they also proved if $u \in L^\infty_{\text{loc}}(\mb R^N)$ is a finite energy solution of \eqref{B-L_equation}, then $u$ satisfies the Poho\v{z}aev identity
\begin{align*}
    \frac{N}{2}\I{\mb R^N} |\na u|^2 dx +\frac{N}{2} \I{\mb R^N}|u|^2 dx = N\I{\mb R^N} H(u) dx,
\end{align*}
see \cite[Proposition~1]{Berestycki-Lions1}. This, in particular, implies that the assumptions on $h$ are "almost necessary for the existence of nontrivial finite energy solutions", \cite[Sec.~2.2, Page 322]{Berestycki-Lions1}. In the spirit of the above work, V. Moroz and J. Van Schaftingen in \cite{Moroz2015existence} studied a class of non-local Choquard equations, precisely:
\begin{equation}\label{Moroz_Schaftingen_problem}
-\Delta u + u = (I_{\alpha}*F(u))F'(u) \text{ in } \mathbb{R}^N,
\end{equation}
with $f$ and $F$ satisfying \ref{f1}-\ref{f3}. Since the nonlocal term is not preserved or controlled under Schwartz symmetrization and the problem is not scale invariant, the approach of Berestycki and Lions fails for the nonlocal term. In order to overcome this difficulty, Moroz and Schaftingen employed a scaling technique introduced by L. Jeanjean \cite{Jeanjean}. More specifically, the authors first constructed a Palais-smale sequence that satisfies asymptotically the Poho\v{z}aev identity and then employed a concentration compactness argument to prove the existence of a nontrivial solution to \eqref{Moroz_Schaftingen_problem}. Further, in order to prove that the obtained solution is a ground state requires a Poho\v{z}aev identity, which the author's obtained by addressing the regularity of the solution. \\
The case of fractional Laplacian is much more involved, since it is difficult to obtain Poho\v{z}aev identity under the minimal assumptions of Berestycki and Lions. Now for the case of local nonlinearities, we refer to the works \cite{Byeon-Kwon-Seok, Chang-Wang}, where the existence of ground state solutions are obtained under stronger regularity assumptions on the nonlinearity viz. $f \in C^1(\mb R) $ in \cite{Byeon-Kwon-Seok} and $f \in C^{0,\mu}_{\text{loc}}(\mb R)$ for some $\mu \in (1-2s,1) $ if $s \in (0,1/2]$ in \cite{Chang-Wang} {respectively}. Regarding the case of nonlocal nonlinearity, we refer to \cite{Cingolani-Gallo-Tanaka-Math-Engg., Luo, Shen-Gao-Yang}. In \cite{Shen-Gao-Yang}, the authors proved a Poho\v{z}aev identity and obtained the existence of ground state solutions under the restriction $\alpha \in (0,2s)$ and $f \in C^1(\mb R)$. Later, Luo in \cite{Luo}, using constrained variational method along with a deformation lemma obtained a ground state solution that satisfies a Poho\v{z}aev identity under the relaxed restriction $\alpha \in \min\{4s,N\}$ and $f \in C(\mb R)$. We remark here that in both the above papers, $f$ is assumed to be superlinear. Recently in \cite{Cingolani-Gallo-Tanaka-Math-Engg.}, the authors further relaxed the superlinear assumption on the nonlinearity and proved the existence of least energy solution on the Poho\v{z}aev manifold for all $\alpha \in (0,N)$. {More precisely, they studied the following 
\begin{equation*}
    (-\Delta)^s u+\mu u = (I_{\alpha}*F(u))F'(u) \text{ in } \mathbb{R}^N,
\end{equation*}
with $f$ satisfying:
\begin{enumerate}
    \item $f\in C(\mathbb{R},\mathbb{R});$
    \item  $\displaystyle \limsup_{t\rightarrow 0}\frac{|tf(t)|}{|t|^{\frac{N+\alpha}{N}}}<+\infty
    $, and $\displaystyle \limsup_{|t|\rightarrow+\infty}\frac{|tf(t)|}{|t|^{\frac{N+\alpha}{N-2s}}}<+\infty$;
    \item defining $F(t):=\int_0^tf(\tau)d\tau$, it satisfies:
    $$\lim_{t\rightarrow 0}\frac{F(t)}{|t|^{\frac{N+\alpha}{N}}}=0 \text{ and } \lim_{|t|\rightarrow +\infty}\frac{F(t)}{|t|^{\frac{N+\alpha}{N-2s}}}=0;$$
    \item there exists $t_0\in \mathbb{R}\setminus \{0\}$, such that $F(t_0)\neq 0$, 
\end{enumerate}
and proved the existence of a radially symmetric weak solution that satisfies the following Poho\v{z}aev identity:
$$\left(\frac{N-2s}{2}\right)[u]^2+\frac{N}{2}\mu \left\| u \right\|_2^2=\left(\frac{N+\alpha}{N}\right)\int_{\mathbb{R}^N}(I_{\alpha}*F(u))F(u)dx,$$
moreover, the weak solution turns out to be the ground state on the Poho\v{z}aev manifold as well.
}

Recently, research is also being carried out to obtain the Poho\v{z}aev identity for the nonlocal operator under the minimal regularity assumptions on the nonlinearity and the solution. We refer to the works \cite{Anthal2025Pohozaev, ambrosio, Cingolani-Gallo-Tanaka-Adv.-Nonlinear-Stud} in this direction. \\
Motivated by the above literature, we are interested to study the mixed {local-nonlocal} problem \eqref{1.1}. An operator of this kind comes into play when both local and nonlocal changes affect a physical phenomenon. This operator may be encountered in the study of bi-model power law distribution processes, as referenced in \cite{pagnini2021should}. It is also examined in the context of optimum searching theory, biomathematics, and animal foraging; see \cite{dipierro2022non} for further elaboration. Various contributions have examined issues pertaining to the existence of solutions, their regularity and symmetry properties, Faber-Krahn type inequalities, Neumann problems, see for instance \cite{biagi2021global, biagi2022mixed, abatangelo2021elliptic, Garain2023Higher} and the references therein. Regarding Choquard equation involving mixed local-nonlocal operators we {refer to} the works \cite{Anthal2023Choquard, Anthal2025Mixed} which
address the case of a bounded domain with specific boundary conditions. Additionally, the case of $\mathbb{R}^N$ has also been addressed in \cite{Giacomoni2025Normalized, Nidhi2025Normalized, Nidhi2025Quasilinear, Nidhi2025Solution}, where the authors studied the existence, nonexistence, regularity and multiplicity of normalized solutions.

	\subsection{Main results}
    Throughout the rest of the paper, we assume that $0<s<1$, $N\geq 3$ and $\alpha\in(0,N)$ unless otherwise mentioned.
   Our first main result reads as follows:
	\begin{theorem}\label{Theorem 1}
		Assume that $f$ satisfies {\rm (f1)--(f3)}. Then \eqref{1.1} has a weak solution in $H^1(\mathbb{R}^N)$.
	\end{theorem}
   \noindent Moreover, we show that the weak solution obtained in Theorem \ref{Theorem 1} above is a ground state solution under the additional hypothesis $(f4)$.
\begin{theorem}\label{Theorem 2}
		Let $f$ satisfy {\rm (f1)--(f4)}, and let $u \in H^1(\mathbb{R}^N)$ be the weak solution obtained in Theorem~\ref{Theorem 1}. Then $u$ is a ground state solution of \eqref{1.1}.
	\end{theorem}
	\noindent Further, we prove the following symmetry result for the ground state solutions.
\begin{theorem}\label{Symmetry}
		Let $f \in C(\mathbb{R},\mathbb{R})$ be an odd function satisfying {\rm (f1)--(f4)}. Assume that $f$ does not change sign on $(0,\infty)$ and that $u \in H^1(\mathbb{R}^N)$ is a positive ground state solution of \eqref{1.1}. Then $u$ is radially symmetric.
	\end{theorem}
  \noindent  We close this section by giving outline of the proofs. As described in \cite{Moroz2015existence}, since our problem is not scale invariant, a constrained minimization as used in \cite{Berestycki-Lions1} cannot be used to study our problem. So, motivated by \cite{Moroz2015existence}, we use the scaling techniques introduced by L. Jeanjean \cite{Jeanjean} to obtain the existence of a nontrivial weak solution to \eqref{1.1} under the assumptions \ref{f1}-\ref{f3} on the nonlinearity $f$.

\noindent Now to conclude that such a constructed solution is a ground    state, a Poho\v{z}aev identity is required, which further require some regularity on the solution. Again motivated by \cite{Moroz2015existence, Nidhi2025Normalized}, we address the issue of regularity by developing a counterpart of the critical Brezis-Kato regularity result for the mixed operator problem with general nonlinearity $f$. Still even after obtaining the $W_\text{loc}^{2,q}(\mb R^N)$, for every $q \geq 2$, we cannot use the standard test function as used in \cite{Moroz2015existence} to obtain the Poho\v{z}aev identity. This is due to the fact that the integration by parts formula for the nonlocal operator require global regularity on the solution see \cite{ambrosio,Cingolani-Gallo-Tanaka-Adv.-Nonlinear-Stud }. To overcome this difficulty, we followed the approach developed in \cite{Anthal2025Pohozaev}. We extend the proof of \cite{Anthal2025Pohozaev} to the case of Choquard type nonlinearity with general nonlinearity. This approach, however requires the additional assumption \ref{f4} on the nonlinearity $f$. These proofs of regularity results and the Poho\v{z}aev identity are the main novelty of this work, which are proved in the Appendix.

   \noindent Finally, motivated by Moroz and Schaftingen \cite{Moroz2015existence} who demonstrated the symmetry of the ground state solution, we prove the symmetry of the solutions . Thanks to the polarization properties associated with the Dirichlet and non-local integrals (refer to \cite[Lemma~5.4 and Lemma~5.5]{Moroz2015existence}), we established a relation between the Gagliardo semi-norm of the solution and its polarization, as demonstrated in Lemma \ref{Updated_Lemma 3.3.3}. This ultimately enabled us to demonstrate the radial symmetry of the ground state solution pertaining to our problem involving the fractional Laplacian operator as well.

\section{Functional setting and auxiliary results}
	We consider the Banach space $H^{1}(\mathbb{R}^N)$ consisting of measurable functions $u:\mathbb{R}^N\to\mathbb{R}$ such that $u$ and its weak derivatives $u_{x_i}$, $i=1,2,\ldots,N$ belongs to $L^2(\mathbb{R}^N)$, which is equipped with the following norm:
	$$\left\| u\right\|=(\left\| \nabla u\right\|_2^2+\left\| u\right\|_2^2+[u]^2)^{\frac{1}{2}},$$
	where  
    $$
    \|u\|_2^{2}:=\int_{\mathbb{R}^N}|u|^2\,dx,\quad \|\na u\|_2^{2}:=\int_{\mathbb{R}^N}|\na u|^2\,dx, 
    $$
    and 
	$$[u]^2:=\frac{C(N,s)}{2}\int_{\mathbb{R}^N}\int_{\mathbb{R}^N}\frac{|u(x)-u(y)|^2}{|x-y|^{N+2s}}\,dx dy.$$
	The notion of weak solution for \eqref{1.1} is as follows:
	\begin{defi}(Weak solution)
		A function $u\in H^{1}(\mathbb{R}^N)$ is said to be a weak solution of \eqref{1.1} if 
		\begin{equation}\label{weak_sol}
			\int_{\mathbb{R}^N} \nabla u\nabla v\,dx  +\ll u, v \gg + 
			\int_{\mathbb{R}^N}uv\,dx  = \int_{\mathbb{R}^N}(I_{\alpha}*F(u))f(u)v \,dx,
		\end{equation}
		for every $v\in H^{1}(\mathbb{R}^N)$,
		where 
		\begin{align}\label{edfp}
			\ll u,v \gg :=\frac{C(N,s)}{2}\int_{\mathbb{R}^N}\int_{\mathbb{R}^N}\frac{(u(x)-u(y))(v(x)-v(y))}{|x-y|^{N+2s}}\,dx dy.
		\end{align}
	\end{defi}

\begin{defi}(Ground state solution)
    We say that $u \in H^1(\mathbb{R}^N) \setminus \{0\}$ is a ground state solution of \eqref{1.1} if $u$ is a weak solution of \eqref{1.1} and
\begin{equation}\label{1.4}
I(u) = m := \inf \left\{ I(v) : v \in H^1(\mathbb{R}^N) \setminus \{0\} \text{ is a solution of } \eqref{1.1} \right\}.
\end{equation}
\end{defi}

	\noindent Defining 
    $$
    A(u):=\int_{\mathbb{R}^N}(I_{\alpha}*F(u))F(u) dx,
    $$
    the energy functional corresponding to the problem \eqref{1.1} is given by
	$$I(u)=\frac{1}{2}\left\|\nabla u \right\|_{2}^2+\frac{1}{2}\left\|  u\right\|_2^2+\frac{1}{2}[u]^2-\frac{1}{2}A(u),$$
	that is, a critical point of $I$ turns out to be the weak solution of \eqref{1.1}. \\
   \noindent The exponents $\frac{N+\alpha}{N}$ and $\frac{N+\alpha}{N-2}$ are called the lower and upper critical exponents respectively, with respect to the following Hardy-Littlewood-Sobolev inequality, see \cite[Theorem~4.3]{lieb2001analysis}:
	\begin{proposition}\label{prop1}
		Let $t,r>1$ and $0<\alpha <N$ with $1/t+1/r=1+\alpha/N$, $f\in L^t(\mathbb{R}^N)$ and $h\in L^r(\mathbb{R}^N)$. There exists a sharp constant $C(t,r,\alpha,N)$ independent of $f$ and $h$, such that
		\begin{equation}\label{co9}
			\int_{\mathbb{R}^N}\int_{\mathbb{R}^N}\frac{f(x)h(y)}{|x-y|^{N-\alpha}}\, {dx dy}\leq C(t,r,\alpha,N) \|f\|_{L^t(\mathbb{R}^N)}\|h\|_{{L^r(\mathbb{R}^N)}}.
		\end{equation}
		If $t=r=2N/(N+\alpha)$, then
		\begin{align}\label{C_alpha}
			C(t,r,\alpha,N)=C(N,\alpha)= \pi^{\frac{N-\alpha}{2}}\frac{\Gamma(\frac{\alpha}{2})}{\Gamma(\frac{N+\alpha}{2})}\left\lbrace \frac{\Gamma(\frac{N}{2})}{\Gamma(N)}\right\rbrace^{-\frac{\alpha}{N}}.
		\end{align}
		Equality holds in  \eqref{co9} if and only if ${f}/{h}\equiv constant$ and
		$\displaystyle h(x)= A(\gamma^2+|x-a|^2)^{-(N+\alpha)/2}$
		for some $A\in \mathbb{C}, 0\neq \gamma \in \mathbb{R}$ and $a \in \mathbb{R}^N$.
	\end{proposition}
     \noindent From the inequality \eqref{co9}, it follows that for any $u\in H^1(\mathbb{R}^N)$
	\begin{align*}
		{\mathcal{A}_q(u):=	\int_{\mathbb R^N}\int_{\mathbb R^N}\frac{|u(x)|^{q}|u(y)|^{q}}{|x-y|^{N-\alpha}}}{ dx dy}
	\end{align*}
	is well defined if $\frac{N+\alpha}{N}\leq q \leq \frac{N+\alpha}{N-2}$.\\
    \noindent It is standard to check using Hardy-Littlewood-Sobolev inequality that if $f \in C(\mb R, \mb R)$ satisfies growth condition \ref{f1}, then the energy functional $I$ defines on the Sobolev space $H^1(\mb R^N)$ is a continuously differentiable functional.
    
\subsection*{Notations}
	We will be using the following notations throughout the {rest of the paper unless otherwise mentioned:}
	\begin{itemize}
		\item For $\Sigma = \{\rho\in C([0,1]; H^1(\mathbb{R}^N)): \rho(0)=0 \text{ and } I(\rho(1))<0\}${\color{blue},} we define $$b := \inf_{\rho\in \Sigma}\sup_{z\in [0,1]}I(\rho(z)).$$
		\item For $u,v \in H^1(\mathbb{R}^N)$, we set 
		\begin{align*}
			\langle u, v\rangle = \int_{\mathbb{R}^N} \nabla u\nabla v\,dx +\ll u, v \gg + 
			\int_{\mathbb{R}^N}uv\,dx
		\end{align*}
		and $\ll \cdot,\cdot\gg$ is defined by \eqref{edfp}.
        \item For $p\geq 1$, $\|\cdot\|_p$ denotes the standard norm on the Banach space $L^p(\mb R^N)$ defined by{
        $$
        \|u\|_p=\left(\int_{\mathbb{R}^N}|u|^p\,dx\right)^\frac{1}{p},\quad \forall u\in L^p(\mathbb{R}^N).
        $$}
        \item $2_s^*=\frac{2N}{N-2s}$ for $0<s\leq 1$ with $N>2s$.

        \item Let $B_R(x)$ denote the open ball of radius $R>0$ centered at $x \in \mathbb{R}^N$.

\item The symbols $C$ and $C_i$ ($i \in \mathbb{N}^+$) denote positive constants whose values may vary from line to line.

\item The symbols $\to$ and $\rightharpoonup$ denote strong convergence and weak convergence, respectively.
	\end{itemize}

\section{Preliminaries for the proof of the main results}\label{Existence}
\subsection{Preliminaries for the proof of Theorem \ref{Theorem 1}}
Define the Poho\v{z}aev functional $P:H^{1}(\mathbb{R}^N)\rightarrow \mathbb{R}$ as:
	\begin{align*}
		P(u):= \left(\frac{N-2}{2}\right)\left\| \nabla u \right\|_2^2+\left(\frac{N-2s}{2}\right)[u]^2+\frac{N}{2}\left\| u \right\|_2^2-\left(\frac{N+\alpha}{2}\right)A(u).
	\end{align*}
	\begin{lemma}\label{Lemma 4.1}
		If $f \in C(\mathbb{R},\mathbb{R})$ satisfies {\rm (f1)} and {\rm (f3)}, then there exists a sequence $\{u_n\}_{n \in \mathbb{N}} \subset H^1(\mathbb{R}^N)$ such that
\begin{enumerate}
  \item\label{Lemma_4.1_1} $I(u_n) \to b > 0$ as $n \to \infty$,
  \item\label{Lemma_4.1_2} $I'(u_n) \to 0$ strongly in the dual space of $H^1(\mathbb{R}^N)$ as $n \to \infty$,
  \item\label{Lemma_4.1_3} $P(u_n) \to 0$ as $n \to \infty$.
\end{enumerate}
	\end{lemma}
	\begin{proof}
		\textbf{Claim 1:} $b<+\infty$.\\
		For $t_0\in \mathbb{R}$ as in \ref{f3}, we set {$\omega:=t_{0}\chi_{B_1(0)}$, where $\chi_{B_1(0)}$ is the characteristic function and $B_1(0)$ is the unit ball with center $0$}. Clearly $w\in L^2(\mathbb{R}^N)\cap L^{2^*}(\mathbb{R}^N)$ and 
		\begin{align*}
			A(\omega)=F(t_0)^2\int_{B_1}\int_{B_1}\frac{A_{N,\alpha}}{|x-y|^{N-\alpha}}\,dx dy>0.
		\end{align*}
		Using {\rm (f1)}, one can easily see that $A$ is continuous in $L^2(\mathbb{R}^N)\cap L^{2^*}(\mathbb{R}^N)$, therefore by density  of $H^1(\mathbb{R}^N)$ in $L^2(\mathbb{R}^N)\cap L^{2^*}(\mathbb{R}^N)$, we can find $v\in H^1(\mathbb{R}^N)$ such that $A(v)>0$. Now, for a fixed $\tau>0$, define $u_{\tau}(x):=v(\frac{x}{\tau})$, which gives us:
		\begin{equation*}
			I(u_{\tau})  =  \frac{\tau^{N-2}}{2}\left\| \nabla v \right\|_2^2+\frac{\tau^{N-2s}}{2}[v]^2+\frac{\tau^N}{2}\left\| v \right\|_2^2-\frac{\tau^{N+\alpha}}{2}A(v) < 0 \text{ for large } \tau>0.
		\end{equation*}
		Thus, there exists some $u_{\tau_0}\in H^1(\mathbb{R}^N)$ such that $I(u_{\tau_0})<0$. Let us define a path in $\Sigma$ with the help of this $u_{\tau_0}$. Define $\rho: [0,1]\rightarrow H^1(\mathbb{R}^N)$ such that $\rho(t):= tu_{\tau_0}$. Clearly $\rho\in \Sigma$  and so $b<+\infty$.\\
		\textbf{Claim 2:} $b>0$.\\
		By Proposition \ref{prop1}, and the hypothesis {\rm (f1)} and the Sobolev inequality, {there exists a constant $C>0$ such that}
		\begin{eqnarray*}
			A(u) & = & \int_{\mathbb{R}^N}(I_{\alpha}*F(u))F(u)\,dx\leq C\left(\int_{\mathbb{R}^N}|F(u)|^{\frac{2N}{N+\alpha}}\,dx\right)^{\frac{N+\alpha}{N}}\\
			& \leq & C \left(\int_{\mathbb{R}^N}||u|^{\frac{N+\alpha}{N}}+|u|^{\frac{N+\alpha}{N-2}}|^{\frac{2N}{N+\alpha}}\,dx\right) ^{\frac{N+\alpha}{N}} \leq C_2 \left(\int_{\mathbb{R}^N}(|u|^2+|u|^{\frac{2N}{N-2}})\,dx\right)^{\frac{N+\alpha}{N}}\\
			& \leq & C\left(\left(\int_{\mathbb{R}^N}|u|^2\,dx\right)^{\frac{N+\alpha}{N}}+\left(\int_{\mathbb{R}^N}|u|^{2^*}\,dx\right)^{\frac{N+\alpha}{N}}\right) \leq C\left(\left\| u \right\|_2^{2(\frac{N+\alpha}{N})}+\left\| \nabla u \right\|_2^{2(\frac{N+\alpha}{N-2})}\right)\\
			& \leq & C\left(\left\| u \right\|_2^{\frac{2\alpha}{N}}\left\| u \right\|_2^2+\left\| \nabla u \right\|_2^{2(\frac{\alpha+2}{N-2})}\left\| \nabla u \right\|_2^2+[u]^2\right). 
		\end{eqnarray*}
		Clearly, we can find $\delta>0$ such that, 
		\begin{equation*}
			A(u)\leq \frac{1}{4}\left(\left\| u \right\|_2^2+\left\| \nabla u \right\|_2^2+[u]^2\right) \text{ for all } u\in H^1(\mathbb{R}^N) \text{ with } \left\| u \right\|^2\leq\delta.
		\end{equation*}
		Thus for all $u\in H^1(\mathbb{R}^N)$ with $\left\| u \right\|^2\leq\delta$,
		\begin{equation}\label{2.1}
			I(u)\geq \frac{1}{4}\left(\left\| u \right\|_2^2+\left\| \nabla u \right\|_2^2+[u]^2\right).
		\end{equation}
		Now, for any $\rho\in \Sigma$, we have:
		$$\left\| \rho(0)\right\|^2=0<\delta < \left\| \rho(1)\right\|^2,$$
		then, by intermediate value theorem, we can find $\bar{\tau}\in (0,1)$ such that 
		\begin{align*}
			\left\| \rho(\bar{\tau})\right\|^2=\left\| \nabla\rho(\bar{\tau})\right\|_2^2+\left\| \rho(\bar{\tau})\right\|_2^2+[\rho(\bar{\tau})]^2=\delta.
		\end{align*}
		Hence by \eqref{2.1}, $I(\rho(\bar{\tau}))\geq \frac{\delta}{4}$, and so
		\begin{align*}
			\frac{\delta}{4}\leq I(\rho(\bar{\tau}))\leq \sup_{\tau\in [0,1]}I(\rho(\tau)).
		\end{align*}
		Since $\rho\in \Sigma$ is arbitrary, we get:
		\begin{align*}
			b\geq \frac{\delta}{4}>0.
		\end{align*}
		Now we define $\Phi :\mathbb{R}\times H^1(\mathbb{R}^N)\rightarrow H^1(\mathbb{R}^N)$ such that $\Phi(t,v)(x)=v(e^{-t}x)$ for all $x\in \mathbb{R}^N$, $t\in \mathbb{R}$ and $v\in H^1(\mathbb{R}^N)$. Then,
		\begin{align*}
			I(\Phi(t,v)) = \frac{e^{(N-2)t}}{2}\left\| \nabla v \right\|_2^2+\frac{e^{Nt}}{2}\left\| v \right\|_2^2+\frac{e^{(N-2s)t}}{2}[v]^2-\frac{e^{(N+\alpha)t}}{2}A(v).
		\end{align*}
		Setting,
		$$\tilde{b}:=\inf_{\tilde{\rho}\in \tilde{\Sigma}}\sup_{\tau\in [0,1]}I(\Phi(\tilde{\rho}(\tau))),$$
		where, $\tilde{\Sigma}=\{ \tilde{\rho}\in C([0,1]; \mathbb{R}\times H^1(\mathbb{R}^N)): \tilde{\rho}(0)=(0,0) \text{ and } {(I\circ\Phi)(\tilde{\rho}(1))<0
        }\},$
		then we claim that $b=\tilde{b}$.
        
        
        Thus, by \cite[Theorem 2.9]{Willem2012Minimax} there exists a sequence $\{(t_n,v_n)\}_{n\in\mathbb{N}}\subset \mathbb{R}\times H^1(\mathbb{R}^N)$ such that
		\begin{equation}\label{2.2}
			\{(I\circ \Phi)(t_n,v_n)\}\rightarrow b \text{ and } (I\circ \Phi)'(t_n,v_n)\rightarrow 0 \text{ as } n\rightarrow \infty.
		\end{equation}
		Now, for any $(t,v)\in \mathbb{R}\times H^1(\mathbb{R}^N)$, we have:
		\begin{eqnarray}\label{2.3}
			0 & = & \lim_{n\rightarrow \infty}(I\circ \Phi)'(t_n, v_n)(t,v)\nonumber\\
			& = & \lim_{n\rightarrow \infty}\left( e^{(N-2)t_n}\int_{\mathbb{R}^N}\nabla v_n\nabla v\,dx+e^{(N-2s)t_n}\ll v_n,v \gg +e^{Nt_n}\int_{\mathbb{R}^N}v_nv\,dx\right.\nonumber\\
			&& -e^{(N+\alpha)t_n}A'(v_n)(v)+t\left(\frac{(N-2)e^{(N-2)t_n}}{2}\left\| \nabla v_n \right\|_2^2 +\frac{(N-2s)e^{(N-2s)t_n}}{2}[v_n]^2\right.\nonumber\\
			&&\left. \left.+\frac{N}{2}e^{Nt_n}\left\| v_n\right\|_2^2-\frac{(N+\alpha)e^{(N+\alpha)t_n}}{2}A(v_n)\right)\right)\nonumber\\
			& = & \lim_{n\rightarrow \infty}\left(I'(\Phi(t_n,v_n))(\Phi(t,v))+ t P(\Phi(t_n,v_n))\right).
		\end{eqnarray}
		For $n \in \mathbb{N}$, set $u_n := \Phi(t_n, v_n)$. Then, by \eqref{2.2} and \eqref{2.3}, we obtain
        $$I(u_n) \to b;I'(u_n) \to 0\text{ strongly in the dual space of } H^1(\mathbb{R}^N);\text{ and }P(u_n) \to 0;$$
as $n \to \infty$.
		This completes the proof.
	\end{proof}
	\noindent Next, we will be using the above-constructed sequence to prove the existence of a weak solution.
	\begin{lemma}\label{Lemma 4.2}
		Let $f \in C(\mathbb{R},\mathbb{R})$ satisfy {\rm (f1)}, {\rm (f2)}, and {\rm (f3)}, and let $\{u_n\}_{n \in \mathbb{N}}$ be the sequence constructed in Lemma~\ref{Lemma 4.1}. Then one of the following alternatives holds:
\begin{enumerate}
  \item[(a)] up to a subsequence, $u_n \to 0$ strongly in $H^1(\mathbb{R}^N)$ as $n \to \infty$;
  \item[(b)] there exist a function $u \in H^1(\mathbb{R}^N) \setminus \{0\}$ such that $I'(u)=0$ and a sequence $\{y_n\}_{n \in \mathbb{N}} \subset \mathbb{R}^N$ satisfying
  \[
  u_n(\cdot - y_n) \rightharpoonup u \quad \text{weakly in } H^1(\mathbb{R}^N)
  \quad \text{as } n \to \infty.
  \]
\end{enumerate}
	\end{lemma}
	\begin{proof}
		Suppose ($a$) does not hold. Then we have 
		\begin{align}\label{e2.4}
			\liminf\limits_{n \rightarrow \infty} \left\| u_n \right\|^2>0. 
		\end{align}
		\textbf{Claim 1:} $\{u_n\}_{n\in\mathbb{N}}$ is bounded in $H^1(\mathbb{R}^N)$.\\
		By Lemma \ref{Lemma 4.1}, since $\{I(u_n)\}_{n\in\mathbb{N}}$ and $\{P(u_n)\}_{n\in\mathbb{N}}$ are convergent, there exists a constant $\beta>0$ independent of $n$ such that
		\begin{eqnarray*}
			\beta & \geq & I(u_n)-\frac{P(u_n)}{(N+\alpha)}= \left(\frac{\alpha+2}{2(N+\alpha)}\right)\left\| \nabla u_n \right\|_2^2+\left(\frac{\alpha+2s}{2(N+\alpha)}\right)[u_n]^2+\frac{\alpha}{2(N+\alpha)}\left\| u_n \right\|_2^2\\
			& \geq & \frac{\alpha}{2(N+\alpha)}\left(\left\| \nabla u_n \right\|_2^2+[u_n]^2+\left\| u_n \right\|_2^2\right)= \frac{\alpha \left\| u_n \right\|^2}{2(N+\alpha)},
		\end{eqnarray*}
        {which proves Claim 1.}\\
		\textbf{Claim 2:} $ \displaystyle\liminf_{n\rightarrow \infty} \sup_{z\in \mathbb{R}^N}\int_{B_1(z)}|u_n|^p\,dx>0$ for all $p\in (2,2^*)$.\\
		If possible, suppose 
		\begin{align}\label{e2.5}
			\liminf_{n\rightarrow \infty} \sup_{z\in \mathbb{R}^N}\int_{B_1(z)}|u_n|^p\,dx=0~\text{ for some}~ p\in (2,2^*).
		\end{align}
		Now, since $\displaystyle \lim_{n\rightarrow \infty}P(u_n)=0$, using \eqref{e2.4}, we have:
		\begin{eqnarray}\label{2.4}
			\liminf_{n\rightarrow \infty} A(u_n) & = & \liminf_{n\rightarrow \infty}\left(\left(\frac{N-2}{N+\alpha}\right)\left\| \nabla u_n \right\|_2^2+\left(\frac{N-2s}{N+\alpha}\right)[u_n]^2+\frac{N}{N+\alpha}\left\| u_n \right\|_2^2\right)\nonumber\\
			& \geq & \liminf_{n\rightarrow \infty}\left(\frac{N-2}{N+\alpha}\right)\left\| u_n \right\|^2>0.
		\end{eqnarray}
		By \cite[Lemma 2.3]{Moroz2013groundstates} and boundedness of $\{u_n\}_{n\in\mathbb{N}}$ in $H^1(\mathbb{R}^N)$, we have:
		\begin{eqnarray*}
			\left\|u_n \right\|_p^p  & \leq & \tilde{C}\left(\sup_{z\in \mathbb{R}^N}\int_{B_1(z)}|u_n|^p\,dx\right)^{\frac{p-2}{p}}\left(\left\| \nabla u_n \right\|_2^2+\left\| u_n \right\|_2^2\right)\\
			& \leq & \tilde{C}\left(\sup_{z\in \mathbb{R}^N}\int_{B_1(z)}|u_n|^p\,dx\right)^{\frac{p-2}{p}}\left(\left\| \nabla u_n \right\|_2^2+\left\| u_n \right\|_2^2+[u_n]^2\right)\\
			& \leq & C' \left(\sup_{z\in \mathbb{R}^N}\int_{B_1(z)}|u_n|^p\,dx\right)^{\frac{p-2}{p}}.
		\end{eqnarray*}
		Also, by \ref{f2} and continuity of $F$, for any $\epsilon>0$ we can find $C_{\epsilon}>0$ such that
		$$|F(t)|^{\frac{2N}{N+\alpha}}\leq \epsilon(|t|^{2}+|t|^{2^*})+C_{\epsilon}|t|^p,$$
		and hence,
		\begin{eqnarray*}
			\liminf_{n\rightarrow \infty}\int_{\mathbb{R}^N}|F(u_n)|^{\frac{2N}{N+\alpha}}\,dx & \leq & \liminf_{n\rightarrow \infty}\left(\epsilon(\left\| u_n \right\|_2^2+\left\| u_n \right\|_{2^*}^{2^*})+C_{\epsilon}\left\| u_n \right\|_p^p\right)\\
			& \leq & \epsilon C_1+C_{\epsilon}C'\liminf_{n\rightarrow \infty} \left(\sup_{z\in \mathbb{R}^N}\int_{B_1(z)}|u_n|^p\,dx\right)^{\frac{p-2}{p}}.
		\end{eqnarray*}
		Since $\epsilon>0$ is arbitrary, using \eqref{e2.5} we get $\displaystyle \liminf_{n\rightarrow \infty}F(u_n)=0$ and hence by Proposition \ref{prop1}, $\displaystyle \liminf_{n\rightarrow \infty}A(u_n)=0$. This is a contradiction to \eqref{2.4}, and hence Claim 2 must hold.
        Thus by \cite[Lemma I.1]{Lions1984concentration} we can find a sequence $\{y_n\}_{n\in\mathbb{N}}\subset \mathbb{R}^N$ and a constant $\delta>0$ such that
        $$\int_{B_1(0)}|u_n(x-y_n)|^2\geq \delta>0.$$

        Now, denoting $v_n:=u_n(.-y_n)$, clearly $\{v_n\}_{n\in\mathbb{N}}$ is bounded in $H^1(\mathbb{R}^N)$ due to the boundedness of $\{u_n\}_{n\in\mathbb{N}}$, and hence there exists $u\in H^1(\mathbb{R}^N)$ such that $v_n \rightharpoonup u$, weakly in $H^1(\mathbb{R}^N)$, up to subsequence. Moreover, since $I$ is continuously differentiable, for any $\phi\in H^1(\mathbb{R}^N)$, if we denote $\phi_n(x):=\phi(x+y_n)$ then
        $$I'(u)(\phi)=\lim_{n\rightarrow \infty}I'(v_n)(\phi)=\lim_{n\rightarrow \infty} I'(u_n)(\phi_n)=0,$$
        by property \ref{Lemma_4.1_2} of Lemma \ref{Lemma 4.1}.
        This completes the proof. 
	\end{proof}
	
\subsection{Preliminaries for the proof of Theorem \ref{Theorem 2}}
In this section, we will be using the Poho\v{z}aev identity to show that the weak solutions obtained in Theorem \ref{Theorem 1} with $f$ satisfying the additional assumption \ref{f4} is a ground state solution of \eqref{1.1}.
	\begin{lemma}\label{lemma 4.3}
		If $f\in C(\mathbb{R};\mathbb{R})$ satisfies \ref{f1} and $u\in H^1(\mathbb{R}^N)\setminus \{0\}$ is a weak solution of \eqref{1.1}, then we can find a path $\rho \in \Sigma$ such that 
		\begin{itemize}
			\item $\rho(1/2)=u;$
			\item $I(\rho(t))<I(u)$ for all $t\in [0,1]\setminus\{1/2\}$.
		\end{itemize}
	\end{lemma}
	\begin{proof}
		Define $\tilde{\rho}:[0,\infty)\rightarrow H^1(\mathbb{R}^N)$ as follows:
		\begin{equation*}
			\tilde{\rho}(t)(x)= \left\{
			\begin{array}{cc}
				u(\frac{x}{t}), & \text{ for } t>0,\\
				0, & \text{ for }t=0,     		
			\end{array} \;\;\;\;\text{ for all } x\in \mathbb{R}^N.
			\right.
		\end{equation*}
		Clearly, $\tilde{\rho}$ is continuous and by \eqref{Poho}
		\begin{eqnarray*}
			I(\tilde{\rho}(t)) & = & \frac{t^{N-2}}{2}\left\| \nabla u \right\|_2^2+\frac{t^{N-2s}}{2}[u]^2+\frac{t^N}{2}\left\| u \right\|_2^2-\frac{t^{N+\alpha}}{2}A(u)\\
			& = & \left(\frac{t^{N-2}}{2}-\frac{(N-2)t^{N+\alpha}}{2(N+\alpha)}\right)\left\| \nabla u \right\|_2^2+\left(\frac{t^{N-2s}}{2}-\frac{(N-2s)t^{N+\alpha}}{2(N+\alpha)}\right)[u]^2\\
			&& +\left(\frac{t^{N}}{2}-\frac{Nt^{N+\alpha}}{2(N+\alpha)}\right)\left\| u \right\|_2^2.
		\end{eqnarray*}
		Defining $g(t):=I(\tilde{\rho}(t))$ for all $t\in [0,\infty)$, one can easily see that a non-zero critical point of $g$ must satisfy:
		$$t^{-\alpha-2}\left(\left(\frac{N-2}{2}\right)\left\| \nabla u \right\|_2^2+\left(\frac{N-2s}{2}\right)t^{2(1-s)}[u]^2+\frac{N}{2}t^2\left\| u \right\|_2^2\right)=K_u,$$
		where $K_u=\left(\frac{N-2}{2}\right)\left\| \nabla u \right\|_2^2+\left(\frac{N-2s}{2}\right)[u]^2+\frac{N}{2}\left\| u \right\|_2^2$.\\
		Claim: $g$ has unique non-zero critical point.\\
		Suppose not, then $h$ must attain $K_u$ atleast twice, where
		$$h(t):=t^{-\alpha-2}\left(\left(\frac{N-2}{2}\right)\left\| \nabla u \right\|_2^2+\left(\frac{N-2s}{2}\right)t^{2(1-s)}[u]^2+\frac{N}{2}t^2\left\| u \right\|_2^2\right)\text{ for all } t>0.$$
		Thus $h$ has atleast one critical point in $(0,\infty)$. But
		\begin{equation*}
			h'(t)  =  -\frac{(\alpha+2)(N-2)}{2t^{\alpha+3}}\left\| \nabla u \right\|_2^2-\frac{(\alpha+2s)(N-2s)}{2t^{\alpha+2s+1}}[u]^2-\frac{\alpha N}{2t^{\alpha+1}}\left\| u \right\|_2^2<0\text{ for all } t>0.
		\end{equation*}
		Thus $g$ has unique critical point in $(0,\infty)$ and since $g'(1)=0$, we can see that $0$ and $1$ are the only critical points of $g$. Now, $g(0)=0$ and 
		$$g(1)= \frac{(\alpha+2)}{2(N+\alpha)}\left\| \nabla u \right\|_2^2+\frac{(\alpha+2s)}{2(N+\alpha)}[u]^2+\frac{\alpha}{2(N+\alpha)}\left\| u \right\|_2^2>0,$$
		therefore, by continuity of $g$, we must have $g(t)<g(1)$ for all $t\in [0,\infty)\setminus \{1\}$, that is 
		\begin{equation}\label{5.1}
			I(\tilde{\rho}(t))<I(u) \text{ for all } t\in [0,\infty)\setminus \{1\}.
		\end{equation}
		Using this $\tilde{\rho}$ we will construct the required path $\rho:[0,1]\rightarrow H^1(\mathbb{R}^N)$ as 
		$$\rho(t)=\tilde{\rho}(2t) \text{ for all } t\in [0,1].$$
		Clearly, $\rho(0)=\tilde{\rho}(0)=0$, $\rho(1/2)=\tilde{\rho}(1)=u$, also by \eqref{5.1}, $I(\rho(t))=I(\tilde{\rho}(2t))<I(u)$ for all $t\in [0,1]\setminus\{1/2\}$ and
		\begin{eqnarray*}
			I(\rho(1)) & = & I(\tilde{\rho}(2))\\
			& = & 2^{N-1}\left(\left(\frac{(N+\alpha)-2^{\alpha+2}(N-2)}{4(N+\alpha)}\right)\left\| \nabla u \right\|_2^2\right.\\
            &&+\left(\frac{(N+\alpha)-2^{\alpha+2s}(N-2s)}{4(N+\alpha)}\right)[u]^2\\
			&& \left.+\left(\frac{(N+\alpha)-2^{\alpha}N}{4(N+\alpha)}\right)\left\| u \right\|_2^2\right) <0. 
		\end{eqnarray*}
		This completes the proof.     \end{proof}

\subsection{Preliminaries for the proof of Theorem \ref{Symmetry}}
\begin{lemma}\label{lemma 4.4}
	{Let $f \in C(\mathbb{R};\mathbb{R})$ satisfy {\rm (f1)} and let $\rho \in \Sigma$. 
If $t_0 \in [0,1]$ and
\[
b := I(\rho(t_0)) > I(\rho(t)) \quad \text{for all } t \in [0,1]\setminus\{t_0\},
\]
then $I'(\rho(t_0)) = 0$.}
\end{lemma}
\begin{proof}
	Let if possible, $I'(\rho(t_0))\neq 0$, then we can find $\epsilon$ and $\delta>0$ such that
	\begin{equation}\label{eq6.1}
		\left\| I'(w)\right\|>\frac{8\epsilon}{\delta} \text{ whenever } { w\in H^1(\mathbb{R}^N) \text{ satisfies }} |I(w)-I(\rho(t_0))|\leq 2\epsilon.
	\end{equation}     
	Hence, by \cite[Lemma 2.3]{Willem2012Minimax} we can find a path $\eta \in C([0,1]\times H^1(\mathbb{R}^N);H^1(\mathbb{R}^N))$ such that:
	\begin{enumerate}
		\myitem{1.}\label{1} $\eta(t,{w})={w}$ if either $t=0$ or $|I({w})-I(\rho(t_0))|>2\epsilon$, { for $t\in [0,1]$ and $w\in H^1(\mathbb{R}^N)$};
		\myitem{2.}\label{2} $I(\eta(1,\rho(t_0)))\leq I(\rho(t_0))-\epsilon=b-\epsilon$; and
		\myitem{3.}\label{3}  $I(\eta(t,{w}))$ is nonincreasing with respect to variable $t$, for every ${w}\in H^1(\mathbb{R}^N)$.
	\end{enumerate}
    {Now, since $b>0$ and $\rho\in \Sigma$, we can find $\epsilon>0$ such that $$2\epsilon< b=|I(\rho(t_0))|=|I(0)-I(\rho(t_0))|=|I(\rho(0))-I(\rho(t_0))|.$$ Then by \ref{1}, $\eta(1,\rho(0))=\rho(0)=0$.}
    Also, by {the properties \ref{1} and \ref{3} above, we have} $I(\eta(1,\rho)(1))\leq I(\eta(0,\rho(1)))=I(\rho(1))<0$. Therefore, $\eta(1,\rho{(1)})\in \Sigma$. Now, since $I(\eta(1,\rho(t_0)))\leq b-\epsilon<b$, and for every $t\in [0,1]\setminus \{t_0\}$, we have
	$$I(\eta(1,\rho(t))) \leq I(\eta(0,\rho(t)))=I(\rho(t))<b.$$
	Thus, by continuity of $I$ and $\eta$, we get
	\begin{equation*}
		b  >  I(\eta(1,\rho(t_0))) =\sup_{\tau\in [0,1]} I(\eta(1,\rho(\tau))) \geq \inf_{\rho\in \Sigma}\sup_{\tau\in [0,1]}I(\rho(\tau)) =b.
	\end{equation*}
	Hence, by contradiction, $I'(\rho(t_0))=0$.
\end{proof}
	\noindent In order to prove radial symmetry of a positive solution, we recall the concept of polarization. For $H\subset \mathbb{R}^N$, a closed half-space, and a function ${g}:\mathbb{R}^N\rightarrow \mathbb{R}$, we define its polarization $g^{H}:\mathbb{R}^N\rightarrow \mathbb{R}$ as follows:
	\begin{equation*}
		g^{H}(x)=\left\{
		\begin{array}{cc}
			\max\{g(x),g(\rho_H(x))\} & \text{ if } x\in H,\\
			\min\{g(x),g(\rho_H(x))\} & \text{ if } x \notin H,
		\end{array}
		\right.
	\end{equation*}
	where $\rho_H$ denotes the reflection about the boundary of $H$. Moreover, we have the following properties \cite{Moroz2015existence}:
	\begin{lemma}\label{Lemma 4.5}
		If $0<u\in H^1(\mathbb{R}^N)$, then $u^H\in H^1(\mathbb{R}^N)$ with $\left\| \nabla u^H \right\|_2^2=\left\| \nabla u \right\|_2^2$ and $\left\| u^H \right\|_2\leq \left\| u \right\|_2$.
	\end{lemma}
    \begin{lemma}\label{Updated_Lemma 3.3.3}
    We have the following polarization relations
    \begin{enumerate}
        \item For $\alpha\in (0,N)$, $0\leq w \in L^{\frac{2N}{N+\alpha}}(\mathbb{R^N})$ we have:
        \begin{equation}\label{New_eq1}
            \int_{\mathbb{R}^N}\int_{\mathbb{R}^N}\frac{w(x)w(y)}{|x-y|^{N-\alpha}}\,dxdy\leq \int_{\mathbb{R}^N}\int_{\mathbb{R}^N}\frac{w^H(x)w^H(y)}{|x-y|^{N-\alpha}}\,dxdy,
        \end{equation}
          with equality if and only if either $w=w^H$ or $w^H=w(\rho_H)$ and hence 
          \begin{equation}\label{New_eq2}
            A(u)\leq A(u^H) \text{ for every } u\in H^1(\mathbb{R}^N),  
          \end{equation}
          with equality if and only if either $u^H=u$ or $u^H=u(\rho_H)$.
        \item For any $w\in H^s(\mathbb{R^N})$, we have:
        $$[w^H]^2\leq [w]^2.$$
    \end{enumerate}
    \end{lemma}
     \begin{proof}
      	As done in the proof of \cite[Lemma 5.3]{Moroz2013groundstates}, for any $\beta>0$ we have:
		\begin{eqnarray}\label{eq 6.3_New}
			\int_{\mathbb{R}^N} \int_{\mathbb{R}^N}\frac{w(x)w(y)}{|x-y|^{\beta}}\,dxdy & = & \int_{H}\int_{H}\frac{(w(x)w(y)+w(\rho_H(x))w(\rho_H(y)))}{|x-y|^{\beta}}\,dxdy\nonumber\\
			&& +\int_{H}\int_{H}\frac{w(x)w(\rho_H(y))+w(\rho_H(x))w(y)}{|x-\rho_H(y)|^{\beta}}\,dxdy\nonumber\\
			& \leq & \int_{H}\int_{H}\frac{(w^H(x)w^H(y)+w^H(\rho_H(x))w^H(\rho_H(y)))}{|x-y|^{\beta}}\,dxdy\nonumber\\
			&& +\int_{H}\int_{H}\frac{w^H(x)w^H(\rho_H(y))+w^H(\rho_H(x))w^H(y)}{|x-\rho_H(y)|^{\beta}}\,dxdy\nonumber\\
			& = & \int_{\mathbb{R}^N}\int_{\mathbb{R}^N} \frac{w^H(x)w^H(y)}{|x-y|^{\beta}}\,dxdy,
		\end{eqnarray}   
        With equality, if and only if either $w^H=w$ or $w^H=w(\rho_H)$, see \cite[Lemma~5.3]{Moroz2013groundstates}, more precisely, if $w^H=w(\rho_H)$, we have:
        \begin{eqnarray*}
            \int_{\mathbb{R}^N}\int_{\mathbb{R}^N}\frac{w(x)w(y)}{|x-y|^{\beta}} & = & \int_{H}\int_{H}\frac{(w(x)w(y)+w(\rho_H(x))w(\rho_H(y)))}{|x-y|^{\beta}}\,dxdy\\
			&& +\int_{H}\int_{H}\frac{w(x)w(\rho_H(y))+w(\rho_H(x))w(y)}{|x-\rho_H(y)|^{\beta}}\,dxdy\\
            & = & \int_{H}\int_{H}\frac{w(x)w(y)+w^H(x)w^H(y)}{|x-y|^{\beta}}\;dxdy\\
            &&+\int_{H}\int_{H}\frac{w(x)w^H(y)+w^H(x)w(y)}{|x-\rho_H(y)|^{\beta}}\;dxdy,
        \end{eqnarray*}
        now, since $w^H(\rho_{H}(x))=w(\rho_H(\rho_H(x)))=w(x)$ for every $x\in \mathbb{R}^N$, we get:
        \begin{eqnarray*}
            \int_{\mathbb{R}^N}\int_{\mathbb{R}^N}\frac{w(x)w(y)}{|x-y|^{\beta}} & = & \int_{H}\int_{H}\frac{(w^H(x)w^H(y)+w^H(\rho_H(x))w^H(\rho_H(y)))}{|x-y|^{\beta}}\,dxdy\\
			&& +\int_{H}\int_{H}\frac{w^H(x)w^H(\rho_H(y))+w^H(\rho_H(x))w^H(y)}{|x-\rho_H(y)|^{\beta}}\,dxdy\\
            & = & \int_{\mathbb{R}^N}\int_{\mathbb{R}^N}\frac{w^H(x)w^H(y)}{|x-y|^{\beta}}\;dxdy.
        \end{eqnarray*}
       Notice that for any $w\in L^{\frac{2N}{N+\alpha}}(\mathbb{R}^N)$, the integrals in \eqref{New_eq1} are well defined and hence by \eqref{eq 6.3_New}, the inequality \eqref{New_eq1} holds. Next, since for every $u\in H^1(\mathbb{R}^N)$, we have:
        $$\int_{\mathbb{R}^N}|F(u)|^{\frac{2N}{N+\alpha}}\leq C\int_{\mathbb{R}^N}\left(|u|^2+|u|^{2^*}\right)<+\infty,$$
        that is, $F(u)\in L^{\frac{2N}{N+\alpha}}(\mathbb{R}^N)$, by \eqref{New_eq1}
        \begin{equation}\label{New_eq3}
          A(u)=\int_{\mathbb{R^N}}\int_{\mathbb{R}^N}\frac{A_{\alpha}F(u)(x)F(u)(y)}{|x-y|^{N-\alpha}}\;dxdy\leq \int_{\mathbb{R^N}}\int_{\mathbb{R}^N}\frac{A_{\alpha}F(u)^H(x)F(u)^H(y)}{|x-y|^{N-\alpha}}\;dxdy.  
        \end{equation}
        Moreover, since $F$ is an increasing function and 
        \begin{equation*}
          F(u)^H(x)=\left\{
          \begin{array}{cc}
              \max\{F(u)(x), F(u)(\rho_H)(x)\} & \text{ if } x\in H,\\
               \min \{F(u)(x), F(u)(\rho_H)(x)\} &  \text{ if } x\notin H, \\
        \end{array} \right.  
        \end{equation*}
        considering all possible cases, one can deduce that $F(u^H)=F(u)^H$. For instance, consider the case when $x\in H$, then 
        $$F(u)^H(x)= \max\{F(u)(x), F(u)(\rho_H)(x)\},$$
        thus, we have following two possibilities:\\
        Either
        $$F(u)^H(x)=F(u)(x)\Rightarrow F(u)(x)\geq F(u)(\rho_H(x))\Rightarrow u(x)\geq u(\rho_H(x)),$$
        since $F$ is an increasing function. Therefore, in this case $u^H(x)=u(x)$, and hence $F(u^H)=F(u)=F(u)^H$. \\
        Or
        $$F(u)^H(x)= F(u)(\rho_H(x))\Rightarrow F(u)(x)\leq F(u)(\rho_H(x))\Rightarrow u(x)\leq u(\rho_H(x)),$$
        since $F$ is an increasing function. Therefore, in this case $u^H(x)=u(\rho_H(x))$, and hence $F(u^H)=F(u(\rho_H))=F(u)^H$. Similarly, the case of $x\notin H$ can be studied. Thus \eqref{New_eq3} becomes:
        \begin{eqnarray*}
            A(u) & \leq & \int_{\mathbb{R}^N}\int_{\mathbb{R}^N}\frac{A_{\alpha}F(u)^H(x)F(u)^H(y)}{|x-y|^{N-\alpha}}\;dxdy\\
            & = & \int_{\mathbb{R}^N}\int_{\mathbb{R}^N}\frac{A_{\alpha}F(u^H)(x)F(u^H)(y)}{|x-y|^{N-\alpha}}\;dxdy=A(u^H) \text{ for every } u\in H^1(\mathbb{R}^N).
        \end{eqnarray*}
        Next we will see the Gagliardo semi-norm. Suppose $w\in H^s(\mathbb{R}^N)$, then by \eqref{eq 6.3_New} we have:
        \begin{eqnarray*}
            [w^H]^2 & = & \int_{\mathbb{R}^N}\int_{\mathbb{R}^N}\frac{|w^H(x)|^2}{|x-y|^{N+2s}}\,dxdy+\int_{\mathbb{R}^N}\int_{\mathbb{R}^N}\frac{|w^H(y)|^2}{|x-y|^{N+2s}}\,dxdy\\
            && -2 \int_{\mathbb{R}^N}\int_{\mathbb{R}^N}\frac{w^H(x)w^H(y)}{|x-y|^{N+2s}}\,dxdy,\\
            & \leq & \int_{\mathbb{R}^N}\int_{\mathbb{R}^N}\frac{|w^H(x)|^2}{|x-y|^{N+2s}}\,dxdy+\int_{\mathbb{R}^N}\int_{\mathbb{R}^N}\frac{|w^H(y)|^2}{|x-y|^{N+2s}}\,dxdy\\
            && -2 \int_{\mathbb{R}^N}\int_{\mathbb{R}^N}\frac{w(x)w(y)}{|x-y|^{N+2s}}\,dxdy,
        \end{eqnarray*}
        since, by HLS-inequality (Proposition \ref{prop1}) the above integrals are well defined, thus by Lemma \ref{Lemma 4.5}
        \begin{eqnarray*}
			[w^H]^2 & \leq & \int_{\mathbb{R}^N}\frac{1}{|z|^{N+2s}}\left(\int_{\mathbb{R}^N}|w^H(y+z)|^2dy\right)dz+
			\int_{\mathbb{R}^N}\frac{1}{|z|^{N+2s}}\left(\int_{\mathbb{R}^N}|w^H(x+z)|^2dy\right)dz\\
			&& -2 \int_{\mathbb{R}^N}\int_{\mathbb{R}^N}\frac{w(x)w(y)}{|x-y|^{N+2s}}dxdy
			=  \int_{\mathbb{R}^N}\frac{2\left\| w^H \right\|_2^2}{|z|^{N+2s}}dz-2 \int_{\mathbb{R}^N}\int_{\mathbb{R}^N}\frac{w(x)w(y)}{|x-y|^{N+2s}}dxdy\\
			& \leq & \int_{\mathbb{R}^N}\frac{2\left\| w \right\|_2^2}{|z|^{N+2s}}dz-2 \int_{\mathbb{R}^N}\int_{\mathbb{R}^N}\frac{w(x)w(y)}{|x-y|^{N+2s}}dxdy=[w]^2.
            \end{eqnarray*}
     \end{proof}
    
	\begin{lemma}\label{Lemma 4.7}
		Assume that $u\in L^2(\mathbb{R}^N)$ is a nonnegative function, then there exists $x_0\in \mathbb{R}^N$ and a non- increasing function $v:(0,\infty)\rightarrow \mathbb{R}$ such that $u(x)=v(|x-x_0|)$ almost everywhere, if and only if, for every closed half-space $H\subset \mathbb{R}^N$ either $u^H=u$ or $u^H=u(\rho_H)$. 
	\end{lemma}
	\begin{proof}
		Suppose $0\leq u \in L^2(\mathbb{R}^N)$ is such that, for every closed half-space $H$, either $u^H=u$ or $u^H=u(\rho_H)$, then by \cite[Lemma 5.4]{Moroz2013groundstates} there exists $x_0\in \mathbb{R}^N$ and a nonincreasing function $v:(0,\infty)\rightarrow \mathbb{R}$ satisfying $u(x)=v(|x-x_0|)$ almost everywhere in $\mathbb{R}^N$. Conversely, for a closed half-space $H\subset \mathbb{R}^N$, let if possible there exists $x,y\in \mathbb{R}^N$ such that $u^H(x)\neq u(\rho_H(x))$ and $u^H(y)\neq u(y)$. Then we have the following cases:\\
		Case 1: $x_0\in H$.\\
		Now, if $y\in H$, then we must have $u(\rho_H(y))>u(y)$ and hence $|\rho_H(y)-x_0|<|y-x_0|$, since $v$ is nonincreasing. But, since $y,x_0\in H$, we have $|y-x_0|<|\rho_H(y)-x_0|$. Similarly, if $y\notin H$, we get $|\rho_H(y)-x_0|>|y-x_0|$ but since $x_0,\rho_H(y)\in H$ we get a contradiction. Thus $u^H=u$.\\
		Case 2: $x_0\notin H$.\\
		Now, if $x\in H$, then we must have $u(\rho_H(x))<u(x)$ and hence $|\rho_H(x)-x_0|>|x-x_0|$, since $v$ is nonincreasing. But, since $\rho_H(x),x_0\notin H$, we have $|\rho_H(x)-x_0|<|x-x_0|$. Similarly, if $x\notin H$, we get $|\rho_H(x)-x_0|<|x-x_0|$ but since $x_0,x\notin H$ we get a contradiction. Thus $u^H=u(\rho_H)$.     	
	\end{proof}

\section{Proof of the main results}

\subsection{Proof of Theorem \ref{Theorem 1}}
		If (a) of Lemma \ref{Lemma 4.2} is true, then by continuity of $I$, we get $b=0$, which is a contradiction to Lemma \ref{Lemma 4.1}. Thus, (a) does not hold. So, (b) of Lemma \ref{Lemma 4.2} holds, which gives existence of $u\in H^1(\mathbb{R}^N)\setminus\{0\}$ such that $I'(u)=0$, therefore $u$ is a weak solution of \eqref{1.1}. This completes the proof.

\subsection{Proof of Theorem \ref{Theorem 2}}
		By  Section \ref{Existence}, we can find a Poho\v{z}aev-Palais-Smale sequence $\{u_n\}_{n\in\mathbb{N}}$ in $H^1(\mathbb{R}^N)$ at level $b$ and $u\in H^1(\mathbb{R}^N)$ such that $u$ is a  weak solution of \eqref{1.1} and $u_n\rightharpoonup u$ weakly in $H^1(\mathbb{R}^N)$. Therefore, by \eqref{Poho} we have:
		\begin{eqnarray*}
			I(u) & = & I(u)-\frac{P(u)}{N+\alpha}\\
			& = & \left(\frac{\alpha+2}{2(N+\alpha)}\right)\left\| \nabla u \right\|_2^2+\left(\frac{\alpha+2s}{2(N+\alpha)}\right)[u]^2+\frac{\alpha}{2(N+\alpha)}\left\| u \right\|_2^2\\
			& \leq & \liminf_{n\rightarrow \infty}\left(\left(\frac{\alpha+2}{2(N+\alpha)}\right)\left\| \nabla u_n \right\|_2^2+\left(\frac{\alpha+2s}{2(N+\alpha)}\right)[u_n]^2+\frac{\alpha}{2(N+\alpha)}\left\| u_n \right\|_2^2\right)\\
			& = & \liminf_{n\rightarrow \infty}\left(I(u_n)+\frac{P(u_n)}{N+\alpha}\right)=\liminf_{n\rightarrow \infty} I(u_n)=b.
		\end{eqnarray*}
		Also, $$I(u)\geq m:=\inf\{I(v): {v\in H^1(\mathbb{R}^N)\setminus\{0\} \text{ is a weak solution of } \eqref{1.1}} \},$$
		therefore,
		\begin{align}\label{e5.2}
			m\leq I(u)\leq b.
		\end{align}
		Now, let $v\in H^1(\mathbb{R}^N)$ be such that $v$ solves \eqref{1.1} and $I(v)\leq I(u)$. Then by Lemma \ref{lemma 4.3}, we can find $\rho\in \Sigma$ such that $\rho(1/2)=v$ and $I(\rho(t))<I(v)$ for all $t\in [0,1]\setminus \{1/2\}$. Thus,
		$$I(u)\leq b\leq I(\rho(t))\leq I(v)\leq I(u),$$
		hence $I(v)=I(u)=b$ for every $v$ that solves \eqref{1.1}. Therefore, by using the definition of $m$ and \eqref{e5.2}, we have $m=I(u)=b$ which implies $u$ is the ground state solution. \qed
\subsection{Proof of Theorem \ref{Symmetry}}
\noindent Using the above properties of polarization, let us study the symmetry of a positive solution to \eqref{1.1}.
	\begin{myproof}{Theorem}{\ref{Symmetry}}
		Without loss of generality, let us assume that $f(x)\geq 0$ for all $x\in (0,\infty)$. Since $u$ {is a weak solution of \eqref{1.1}}, by Lemma \ref{lemma 4.3}, we can find a path $\tilde{\sigma}\in \Sigma$ such that $\tilde{\sigma}(1/2)=u$ and $I(\tilde{\sigma}(t))<I(u)$ for all $t\in [0,1]\setminus\{1/2\}$. Taking $\sigma=|\tilde{\sigma}|$, we get $\sigma\in\Sigma$ with $\sigma(1/2)=u$ and $\sigma(t)\geq 0$ for all $t\in [0,1]$. Now, for a closed half-space $H$ {of $\mathbb{R}^N$, we} define $\sigma^H:[0,1]\rightarrow H^1(\mathbb{R}^N)$ such that 
		$$\sigma^H(z)=\sigma(z)^H \text{ for all } z\in [0,1].$$
		Clearly, $\sigma^H\in C([0,1];H^1(\mathbb{R}^N)$, by continuity of $\sigma$ and polarization. Now, as explained in the proof of Lemma \ref{Updated_Lemma 3.3.3}, we have $F(w)^H=F(w^H)$ for every $w\in H^1(\mathbb{R}^N)$.
        Thus, by Lemma \ref{Lemma 4.5} and Lemma \ref{Updated_Lemma 3.3.3} we have:
		\begin{eqnarray}\label{eq 6.4}
			I(\sigma^H(z)) & \leq & \frac{\left\|\nabla \sigma(z)\right\|_2^2}{2}+\frac{\left\| \sigma(z)\right\|_2^2}{2}+\frac{[\sigma(z)]^2}{2}-\int_{\mathbb{R}^N}(I_{\alpha}*F(\sigma(z))^H)F(\sigma(z))^H\,dx\nonumber\\
			& \leq &  \frac{\left\|\nabla \sigma(z)\right\|_2^2}{2}+\frac{\left\| \sigma(z)\right\|_2^2}{2}+\frac{[\sigma(z)]^2}{2}-\int_{\mathbb{R}^N}(I_{\alpha}*F(\sigma(z)))F(\sigma(z))\,dx\nonumber\\
			& = &  I(\sigma(z)) \text{ for all } z\in [0,1].
		\end{eqnarray}
		This gives us that $\sigma^H\in \Sigma$ and hence 
		$$b\leq \max_{z\in [0,1]}I(\sigma^H(z)).$$
		Also, since $I(|v|)\leq I(v)$, for all $z\in [0,1]$ we get:
		$$I(\sigma^H(z))\leq I(\sigma(z))=I(|\tilde{\sigma}(z)|)\leq I(\tilde{\sigma}(z))\leq I(u)=b,$$
		therefore,
		$\displaystyle\max_{z\in [0,1]}I(\sigma^H(z))\leq b$ and hence $b=\displaystyle\max_{z\in [0,1]}I(\sigma^H(z))$.  Now, since $u$ is a ground state solution, for all $z\in [0,1]$
		$$I(\sigma^H(1/2)) = I(u^H) \geq I(u) \geq I(\tilde{\sigma}(z))\geq I(|\tilde{\sigma}(z)|)= I(\sigma(z))\geq I(\sigma^H(z)).$$
		Thus, 
		$$I(u)=b =\max_{z\in [0,1]}I(\sigma^H(z))=I(\sigma^H(1/2))=I(u^H).$$
		By Lemma \ref{Lemma 4.5} and Lemma \ref{Updated_Lemma 3.3.3}, we can see that $I(u)=I(u^H)$ implies that either $F(u)^H=F(u)$ or $F(u)^H=F(u)(\rho_H)$. Now, if $F(u)^H=F(u)$, then for all $x\in H$ we have $F(u)(x)\geq F(u)(\rho_H(x))$, and hence
		$$\int_{u({\rho_H(x)})}^{u(x)}f(s)ds=F(u)(x)-F(u)(\rho_H(x))\geq 0,$$
		this implies that, for all $x\in H$, either $u(x)\geq u(\rho_H(x))$ or $f(s)=0$ for all $s\in [u(x),u(\rho_H(x))]$. Therefore, $f(u^H)(x)=f(u)(x)$ for all $x\in H$, similarly, one can see the case of $x\in \mathbb{R}^N\setminus H$ and deduce that $f(u^H)=f(u)$ in $\mathbb{R}^N$. Now, since $b=I(u^H)=I(\sigma^H(1/2))>I(\sigma^H(z))$ for all $z\in [0,1]\setminus\{1/2\}$, we can use Lemma \ref{lemma 4.4} to conclude that $u^H$ is a weak solution of
		$$-\Delta u^H+(-\Delta)^s u^H+u^H=(I_{\alpha}*F(u^H))f(u^H)=(I_{\alpha}*F(u))f(u) \text{ in } \mathbb{R}^N,$$
		and since $u$ solves \eqref{1.1}, we observe that $u^H=u$. Similarly, for the case of $F(u)^H=F(u)(\rho_H)$, we can deduce that $u^H=u(\rho_H)$. Since, $H$ is an arbitrary closed half-space {of $\mathbb{R}^N$}, by Lemma \ref{Lemma 4.7}, $u$ is radially symmetric.
	\end{myproof}
\section{Appendix: Regularity and Poho\v{z}aev identity}
    In this section, we prove the regularity results and Poho\v{z}aev identity stated as follows:
    \begin{mytheorem}\label{L_infinity}
		{Let $u \in H^1(\mathbb{R}^N)$, $u \not\equiv 0$, be a weak solution of \eqref{1.1}, where $f$ satisfies {\rm (f1)--(f3)}. 
Then $u \in L^q(\mathbb{R}^N)$ for all $2 \le q \le \infty$. 
Moreover, $u \in W^{2,q}_{\mathrm{loc}}(\mathbb{R}^N)$ for all $q \ge 2$, and
\[
u \in C^{1,\delta}_{\mathrm{loc}}(\mathbb{R}^N) \quad \text{for every } \delta \in (0,1).
\]}
	\end{mytheorem}
    \noindent Using this regularity result, we then obtain the following Poho\v{z}aev identity.
    \begin{mytheorem}\label{Pohozaev_identity}
    If $u \in H^1(\mathbb R^N)$ is a weak solution of \eqref{1.1} with $f$ satisfying {\rm (f1)--(f4)}, then it satisfies
    \begin{align}\label{Poho}
    \left(\frac{N-2}{2}\right)\|\nabla u\|_2^2 + \left(\frac{N-2s}{2}\right)[u]^2+\frac{N}{2}\|u\|_2^2 =\left(\frac{N+\alpha}{2}\right)A(u).
    \end{align}
\end{mytheorem}
    
   \noindent Regarding the regularity result, thanks to Lemma 3.5 of \cite{Giacomoni2020regularity} we can follow the work of \cite{Moroz2015existence} to get higher order integrability of the solution. Precisely, we have the following:
	\begin{mylemma}\label{lemma_6.1}
		Let $u\in H^1(\mathbb{R}^N)$ be a weak solution of \eqref{1.1} with $f$ satisfying \ref{f1}-\ref{f3}. Then $u\in L^q(\mathbb{R}^N)$ for all $q\in [2,\frac{N}{\alpha}2^*)$.
	\end{mylemma}
	\noindent In order to prove above result, we require the following proposition which we take from \cite[Propsition~2.1]{Giacomoni2025Normalized}.
	\begin{proposition}\label{prop6.1}
		If $H$, $K \in L^{\frac{2N}{\alpha}}(\mathbb{R}^N)+L^{\frac{2N}{\alpha+2}}(\mathbb{R}^N)$ and $u\in H^1(\mathbb{R}^N)$ solves
		\begin{equation}\label{H-K}
			-\Delta u +(-\Delta)^su+u= (I_{\alpha}*Hu)K \text{ in } \mathbb{R}^N,
		\end{equation}
		then $u\in L^p(\mathbb{R}^N)$ for all $p\in[2,\frac{N}{\alpha}2^*)$.
	\end{proposition} 
	\noindent	Using the above proposition and growth of $F$ we can prove Lemma \ref{lemma_6.1} as follows:
	\begin{myproof}{Lemma}{\ref{lemma_6.1}}
		Define $H,K:\mathbb{R}^N\rightarrow \mathbb{R}$ as:
		$$H(x):=\frac{F(u(x))}{u(x)}  \text{ and } K(x):=f(u(x)).$$
		By \ref{f1} we have:
		$$|H(x)|\leq C\left(\frac{N}{N+\alpha}|u(x)|^{\frac{\alpha}{N}}+\frac{N-2}{N+\alpha}|u(x)|^{\frac{\alpha+2}{N-2}}\right), $$
        and 
        $$|K(x)|\leq C\left(|u(x)|^{\frac{\alpha}{N}}+|u(x)|^{\frac{\alpha+2}{N-2}}\right),$$
		thus $H,K\in L^{\frac{2N}{\alpha}}(\mathbb{R}^N)+L^{\frac{2N}{\alpha+2}}(\mathbb{R}^N)$. Applying Proposition \ref{prop6.1}, we are done with the proof.
	\end{myproof}
\begin{myproof}{Theorem}{\ref{L_infinity}}
		Let $0\neq u \in H^1(\mathbb{R}^N)$ be a weak solution of \eqref{1.1}, then by \ref{f1} we have:
	\begin{align*}
		(I_{\alpha}*F(u))(x) & =  \int_{\mathbb{R}^N}\frac{A_{\alpha}F(u(x-y))}{|y|^{N-\alpha}}dy\\
		&\leq C_{\alpha}\left(\int_{\mathbb{R}^N}\frac{|u(x-y)|^{\frac{N+\alpha}{N}}}{|y|^{N-\alpha}} dy+\int_{\mathbb{R}^N}\frac{|u(x-y)|^{\frac{N+\alpha}{N-2}}}{|y|^{N-\alpha}} dy\right).
	\end{align*}
	Now, for $\beta\in \{ \frac{N+\alpha}{N}, \frac{N+\alpha}{N-2}\}$, by H$\ddot{\text{o}}$lder's inequality and Lemma \ref{lemma_6.1} we have:
	\begin{eqnarray*}
		\int_{\mathbb{R}^N}\frac{|u(x-y)|^{\beta}}{|y|^{N-\alpha}}dy & = & \int_{B_1}\frac{|u(x-y)|^{\beta}}{|y|^{N-\alpha}}dy+\int_{\mathbb{R}^N\setminus B_1}\frac{|u(x-y)|^{\beta}}{|y|^{N-\alpha}}dy\\
		& \leq & C_1\left(\int_{B_1}\frac{dy}{|y|^{\frac{(N-\alpha)\gamma}{\gamma-1}}}\right)^{\frac{\gamma-1}{\gamma}}+C_2\left(\int_{\mathbb{R}^N\setminus B_1}\frac{dy}{|y|^{2N}}\right)^\frac{N-\alpha}{2N}<M,
	\end{eqnarray*}
	for some $M>0$ whenever  $1<\frac{N}{\alpha}<\gamma < \frac{2^*N}{\alpha\beta}$. Therefore, 
	\begin{equation}\label{eq I_alpha}
		(I_{\alpha}*F(u))\in L^{\infty}(\mathbb{R}^N),
	\end{equation}
		and hence, there exists some constant $M_1>0$ such that $(I_{\alpha}*F(u))(x)\leq M_1$ for every $x\in \mathbb{R}^N$. By \ref{f1} we get:
		\begin{equation*}
			|(I_{\alpha}*F(u))f(u)|\leq M_1 C\left(|u|^{\frac{\alpha}{N}}+|u|^{\frac{\alpha+2}{N-2}}\right).
		\end{equation*}
		Now, if $|u(x)|\leq 1$, we have $|u|^{\frac{\alpha}{N}}+|u|^{\frac{\alpha+2}{N-2}}\leq 2 \leq 2(1+|u|^{\frac{N+2}{N-2}})$, else $|u(x)|>1$ and hence $|u|^{\frac{\alpha}{N}}+|u|^{\frac{\alpha+2}{N-2}}<2|u|^{\frac{N+2}{N-2}}\leq 2(1+|u|^{\frac{N+2}{N-2}})$. Thus
		\begin{equation}\label{eq_4.3}
			|(I_{\alpha}*F(u))f(u)|\leq 2M_1C (1+|u|^{2^*-1})=C_1(1+|u|^{2^*-1}).
		\end{equation}
	{Now motivated by \cite{Su-Valdinoci-Wei-Zhang}, we prove the $L^\infty(\mb R^N)$ estimate as follows: }for $\beta>1$ and $k>0$, let us define:
	\begin{equation}\label{Phi_k}
		\Phi_k(t):=\left\{
		\begin{array}{ll}
			-\beta k^{\beta-1}(t+k)+k^{\beta} & \text{ for } t\leq k, \\
			|t|^{\beta} & \text{ for } -k < t\leq k,\\
			\beta k^{\beta-1}(t-k)+k^{\beta} & \text{ for } k< t.
		\end{array}
		\right.
	\end{equation}
	Clearly $\Phi$ is convex, Lipschitz and satisfies the following:
	\begin{equation}\label{prop_of_Phi_k}
		\Phi_k(t)\leq |t|^{\beta};\;\; |\Phi_k'(t)|\leq\beta |t|^{\beta-1};\;\;\Phi_k(t)\leq t\Phi_k'(t) \leq \beta\Phi_k(t).
	\end{equation}
	Now, for any $\psi\in H^1(\mathbb{R}^N)$, by convexity of $\Phi_k$ we have:
	\begin{eqnarray}\label{eq_4.6}
		\ll \Phi_k(u), \psi\gg & = & \frac{C(N,s)}{2}\int_{\mathbb{R}^N}\int_{\mathbb{R}^N} \frac{(\Phi_k(u(x))-\Phi_k(u(y)))(\psi(x)-\psi(y))}{|x-y|^{N+2s}}dxdy\nonumber\\
		& \leq & \frac{C(N,s)}{2}\int_{\mathbb{R}^N}\int_{\mathbb{R}^N}\frac{(u(x)-u(y))(\psi(x)\Phi_k'(u(x))-\psi(y)\Phi_k'(u(y)))}{|x-y|^{N+2s}}dxdy\nonumber\\
		& = & \ll u, \psi\Phi_k'(u)\gg,
	\end{eqnarray}
	since $\psi(x)\Phi_k'(u(x))(u(x)-u(y))\geq \psi(x)(\Phi_k(u(x))-\Phi_k(u(y)))$ (see \cite[Lemma~A.1]{Brasco2016second}). Replacing $\psi$ by $\Phi_k(u)$ in \eqref{eq_4.6}, by \cite[Theorem~6.5]{Nezza2012Hitchhiker} there exists a constant $C_2>0$
	\begin{equation}\label{eq_4.7}
		\left\|\Phi_k(u)\right\|_{2^*_s}^{2}\leq C_2 [\Phi_k(u)]^2=C_2 \ll \Phi_k(u), \Phi_k(u)\gg\leq \ll u, \Phi_k(u)\Phi_k'(u)\gg,
	\end{equation}
	also, 
    {since $\Phi_K(t)\Phi_K''(t)\geq 0$ almost everywhere,}
	\begin{align}\label{eq_4.8}\nonumber
		\int_{\mathbb{R}^N}\nabla u \nabla(\Phi_k(u)\Phi_k'(u))\,dx & =  \int_{\mathbb{R}^N}\Phi_k'(u)^2|\nabla u|^2\,dx+\int_{\mathbb{R}^N}\Phi_k(u)\Phi_k''(u)|\nabla u|^2\,dx\\
        &\geq \int_{\mathbb{R}^N}\Phi_k'(u)^2|\nabla u|^2.
	\end{align}
	Thus by \eqref{eq_4.8}, \eqref{eq_4.7}, \eqref{prop_of_Phi_k}, and taking $\Phi_k'(u)\Phi_k(u)$ as test function {in  \eqref{weak_sol}}, 
	\begin{align}\label{eq_4.9}\nonumber
		\int_{\rnn}|\nabla u|^2|\Phi_k'(u)|^2 \,dx&+\int_{\rnn} |\Phi_k(u)|^2\,dx\nonumber\\
		 & \leq  \int_{\rnn}|\nabla u|^2|\Phi_k'(u)|^2\,dx+\int_{\rnn}u \Phi_k'(u)\Phi_k(u)\,dx\nonumber\\
		 &\leq  \int_{\rnn}\nabla u \nabla(\Phi_k(u)\Phi_k'(u))\,dx+\int_{\rnn} u \Phi_k'(u)\Phi_k(u)\,dx\nonumber\\
		& = \int_{\rnn}(I_{\alpha}*F(u))f(u)\Phi_k(u)\Phi_k'(u)\,dx-\ll u, \Phi_k'(u)\Phi_k(u)\gg \nonumber\\
		& \leq \int_{\rnn}(I_{\alpha}*F(u))f(u)\Phi_k(u)\Phi_k'(u)\,dx -\left\| \Phi_k(u) \right\|_{2^*_s}^2\nonumber\\
        &\leq \int_{\rnn}(I_{\alpha}*F(u))f(u)\Phi_k(u)\Phi_k'(u)\,dx.
	\end{align}
	Now, by the continuous embedding $H^1(\mathbb{R}^N)\hookrightarrow L^{2^*}(\mathbb{R}^N)$, \eqref{eq_4.3}, \eqref{prop_of_Phi_k} and \eqref{eq_4.9} there exists $C_3>0$ such that
	\begin{align}\label{eq_4.10}\nonumber
		\left\| \Phi_k(u)\right\|_{2^*}^2 \leq & C_3 \left(\left\| \nabla \Phi_k(u)\right\|_2^2+\left\| \Phi_k\right\|_2^2\right)= C_3\left(\int_{\rnn} |\Phi_k'(u)|^2|\nabla u|^2\,dx+\int_{\rnn}|\Phi_k(u)|^2\,dx\right)\\ \nonumber
		 \leq & C_3 \int_{\rnn}(I_{\alpha}*F(u))f(u)\Phi_k(u)\Phi_k'(u)\,dx\\
          \leq &C_3C_1\int_{\rnn}(1+|u|^{2^*-1})|\Phi_k(u)||\Phi_k'(u)|\,dx\\
	 \leq & C_1C_3 \beta\int_{\rnn}|u|^{2\beta-1}\,dx+ C_1C_3\beta\int_{\rnn}|\Phi_k(u)|^2|u|^{2^*-2}\,dx,
	\end{align}
	here 
	\begin{eqnarray*}
		\int_{\rnn}|\Phi_k(u)|^2|u|^{2^*-2}\,dx & = & \int_{\{|u|\leq k\}} |\Phi_k(u)|^2|u|^{2^*-2}\,dx+ \int_{\{|u|> k\}}|\Phi_k(u)|^2|u|^{2^*-2}\,dx\\
		& \leq & \int_{\{|u|\leq k\}}|u|^{2\beta +2^*-2}\,dx+C_4\int_{\{|u|>k\}} |u|^2|u|^{2^*-2}\,dx\\
		& \leq & k^{2\beta-2}\left\| u \right\|_{2^*}^{2^*}+C_4\left\| u \right\|_{2^*}^{2^*}<+\infty,
	\end{eqnarray*}
	since $\Phi_k$ is a linear function for $|t|>k $. Thus, taking $\beta=\beta_1:=\frac{2^*+1}{2}$, the right hand side of \eqref{eq_4.10} is well defined. Now, for any $R>0$, by H$\ddot{\text{o}}$lder's inequality
	\begin{eqnarray*}
		\int_{\rnn}|\Phi_k(u)|^2|u|^{2^*-2}\,dx & = & \int_{|u|\leq R}\frac{|\Phi_k(u)|^2|u|^{2^*-1}}{|u|}\,dx+\int_{\{|u|>R\}}|\Phi_k(u)|^2|u|^{2^*-2}\,dx\\
		& \leq & R^{2^*-1}\int_{\rnn}\frac{\Phi_k(u)^2}{|u|}\,dx\\
        &&+\left(\int_{\{|u|>R\}}|\Phi_k(u)|^{2^*}\,dx\right)^{\frac{2}{2^*}}\left(\int_{\{|u|>R\}}|u|^{2^*}\,dx\right)^{\frac{2^*-2}{2^*}},
	\end{eqnarray*}
	taking $R>0$ large enough, so that 
	$$\left(\int_{\{|u|>R\}}|u|^{2^*}\,dx\right)^{\frac{2^*-2}{2^*}}\leq \frac{1}{2\beta C_1C_3},$$
	since $\Phi_k(u)\leq |u|^{\beta_1}$ by \eqref{eq_4.10}, we get:
	\begin{equation*}
		\left\| \Phi_k(u)\right\|_{2^*}^2 \leq  2C_1C_3\beta\left(\left\| u \right\|_{2^*}^{2^*}+R^{2^*-1}\int_{\rnn}|u|^{2\beta_1-1}\,dx\right)= 2C_1C_3\beta(1+R^{2^*-1})\left\| u \right\|_{2^*}^{2^*}.
	\end{equation*}	
	Now, taking $k\rightarrow \infty$, clearly, $\Phi_k(t)\rightarrow |t|^{\beta_1}$, thus by Fatou's lemma, we have:
	\begin{equation*}
		\left\| u \right\|_{2^*\beta_1}^{2\beta_1}=\left(\int_{\rnn}|u|^{2^*\beta_1}\,dx\right)^{\frac{2}{2^*}}\leq 2C_1C_3\beta(1+R^{2^*-1})\left\| u \right\|_{2^*}^{2^*}<+\infty.
	\end{equation*}
	Therefore \begin{equation*}
		u\in L^{2^*\beta_1}(\mathbb{R}^N).
	\end{equation*}
	Now, suppose $\beta>\beta_1$, then using $\Phi_k(u)\leq |u|^{\beta}$ in \eqref{eq_4.10}, we have:
	\begin{equation}\label{eq_4.11}
		\left\| \Phi_k(u) \right\|_{2^*}^2 \leq C_1C_3\beta \left(\int_{\rnn}|u|^{2\beta-1}\,dx+\int_{\rnn}|u|^{2\beta+2^*-2}\,dx\right).
	\end{equation}
	Taking $A=\frac{2(\beta-1)}{2^*-1}>1$ and $B=\frac{2(\beta-1)}{2(\beta-1)-2^*+1}>1$; and using Young's inequality, we get:
	\begin{eqnarray*}
		\int_{\rnn}|u|^{2\beta-1}\,dx & = & \int_{\rnn}|u|^{\frac{2^*}{A}}|u|^{2\beta-1-\frac{2^*}{A}}\,dx\\
		& \leq & \int_{\rnn}\frac{\left(|u|^{\frac{2^*}{A}}\right)^{A}}{A}\,dx+ \int_{\rnn}\frac{\left(|u|^{2\beta-2-\frac{2^*}{A}}\right)^B}{B}\,dx\\
        &=&
		\int_{\rnn}\frac{|u|^{2^*}}{A}\,dx+\int_{\rnn}\frac{|u|^{2\beta+2^*-2}}{B}\,dx\\
		& \leq & \left\| u \right\|_{2^*}^{2^*}+\int_{\rnn}|u|^{2\beta+2^*-2}\,dx,
	\end{eqnarray*}
	thus \eqref{eq_4.11} becomes:
	\begin{equation}
		\left\| \Phi_k(u)\right\|_{2^*}^2 \leq C_1C_3\beta \left(\left\| u \right\|_{2^*}^{2^*}+2\int_{\rnn}|u|^{2\beta+2^*-2}\,dx\right)\leq C_4 \beta\left(1+\int_{\rnn}|u|^{2\beta+2^*-2}\,dx\right),
	\end{equation}
	then, by Fatous's lemma
	\begin{equation*}
		\left\| u \right\|_{2^*\beta}^{2\beta} = \left(\int_{\rnn} |u|^{2^*\beta}\,dx\right)^{\frac{2}{2^*}}  \leq C_4 \beta\left(1+\int_{\rnn}|u|^{2\beta+2^*-2}\,dx\right),
	\end{equation*}
	and hence 
	\begin{equation}\label{eq_4.13}
		\left(1+\int_{\rnn}|u|^{2^*\beta}\right)^{\frac{1}{2^*(\beta-1)}}\,dx\leq  (C_5 \beta)^{\frac{1}{2(\beta-1)}}\left(1+\int_{\rnn}|u|^{2\beta+2^*-2}\,dx\right)^{\frac{1}{2^*(\beta-1)}\,}.
	\end{equation}
	Let us define the sequence $\{\beta_n\}_{n\in\mathbb{N}}$ such that $2\beta_{n+1}+2^*-2=2^*\beta_{n}$, which implies that
	$$\beta_{n+1}-1=\left(\frac{2^*}{2}\right)^{n}\left(\beta_1-1\right),$$
	with $\beta_1=\frac{2^*+1}{2}$.
	Replacing $\beta$ by $\beta_{n+1}$ in \eqref{eq_4.13} we have
	\begin{eqnarray*}
		\left(1+\int_{\rnn}{|u|^{2^*\beta_{n+1}}}\,dx\right)^{\frac{1}{2^*(\beta_{n+1}-1)}}& \leq & (C_5\beta_{n+1})^{\frac{1}{2(\beta_{n+1}-1)}}\left(1+\int_{\rnn}|u|^{2^*\beta_n}\,dx\right)^{\frac{1}{2(\beta_{n+1}-1)}}\\
		& = & (C_5\beta_{n+1})^{\frac{1}{2(\beta_{n+1}-1)}}\left(1+\int_{\rnn}|u|^{2^*\beta_n}\,dx\right)^{\frac{1}{2^*(\beta_{n}-1)}}.
	\end{eqnarray*}
	Denoting $\tau'_n:=C_5 \beta_{n}$ and $A_n:=\left(1+\int_{\rnn}|u|^{2^*\beta_{n}}\right)^{\frac{1}{2^*(\beta_{n}-1)}}$, we get:
	\begin{eqnarray*}
		A_{n+1}  \leq  (\tau'_{n+1})^{\frac{1}{2(\beta_{n+1}-1)}}A_n & \leq & (\tau'_{n+1})^{\frac{1}{2(\beta_{n+1}-1)}}(\tau'_n)^{\frac{1}{2(\beta_{n}-1)}}A_{n-1}\\
        &&\cdots \leq \prod_{i=2}^{n+1}\gamma_i^{\frac{1}{2(\beta_i-1)}}A_1\leq C_0 A_1<+\infty,
	\end{eqnarray*}
	since $u\in L^{2^*\beta_1}(\mathbb{R}^N)$, and $\frac{1}{\beta_i -1}=\left(\frac{2}{2^*}\right)\frac{1}{(\beta_1-1)}<\frac{1}{2}$. 
    This gives that $u\in L^\infty(\mathbb{R}^N)$.

	Next we will see that $u\in L^q(\mathbb{R}^N)$ for all $q\geq 2$ as well. Recall \eqref{eq_4.10}, since $\Phi_k(u)\leq |u|^{\beta}$, then as done above, using Young's inequality we have:
	\begin{eqnarray*}
		\left\| \Phi_k(u) \right\|_{2^*}^2 & \leq & C_1C_3\beta \left(\int_{\rnn}|u|^{2\beta-1}\,dx+\int_{\rnn}|u|^{2\beta+2^*-2}\,dx\right)\\
		& = & C_1C_3\beta \left(\int_{\rnn}|u|^{\frac{2^*}{A}}|u|^{2\beta-1-\frac{2^*}{A}}\,dx+\int_{\rnn}|u|^{2\beta+2^*-2}\,dx\right)\\
		& \leq & C_1C_3\beta \left(\frac{1}{A}\int_{\rnn}|u|^{2^*}\,dx+\frac{1}{B}\int_{\rnn}|u|^{2\beta+2^*-2}\,dx+\int_{\rnn}|u|^{2\beta+2^*-2}\,dx\right)\\
		& \leq & C_6\left(1+ \int_{\rnn}|u|^{2\beta+2^*-2}\,dx\right),
	\end{eqnarray*}
	hence, by Fatou's lemma
	\begin{equation}\label{eq_4.14}
		\left\| u \right\|_{2^*\beta} ^{2^*\beta} \leq \bar{C_{\beta}}\left(1+\int_{\rnn}|u|^{2\beta+2^*-2}\,dx\right)^{\frac{2^*}{2}} \text{ for all } \beta>1.
	\end{equation}
	Now, let $\gamma_0:=\frac{N}{\alpha}2^*$, then by Lemma \ref{lemma_6.1} we know that $u\in L^q(\mathbb{R}^N)$ for all $q\in [2,\gamma_0)$ and hence
	$$\int_{\rnn}|u|^{2\beta+2^*-2}\,dx<+\infty \text{ for all } 1 < \beta < \frac{\gamma_0+2-2^*}{2},$$
	thus by \eqref{eq_4.14}, $u\in L^q(\mathbb{R}^N)$ for all $q\in \left(2^*,2^*\left(\frac{\gamma_0+2-2^*}{2}\right)\right)$. Taking $\gamma_1:=2^*\left(\frac{\gamma_0+2-2^*}{2}\right)>\gamma_0$, now we know that $u\in L^q(\mathbb{R}^N)$ for all $q\in [2,\gamma_1)$. Moving on in similar way, now
	$$\int_{\rnn}|u|^{2\beta+2^*-2}\,dx<+\infty \text{ for all } 1<\beta< \frac{\gamma_1+2-2^*}{2},$$
	and hence by \eqref{eq_4.14}, $u\in L^q(\mathbb{R}^N)$ for all $q\in [2, \gamma_2)$, where $\gamma_2=2^*\left(\frac{\gamma_1+2-2^*}{2}\right)>\gamma_1$. Constructing the sequence $\{\gamma_n\}_{n\in\mathbb{N}}$ such that $\gamma_0$ is as defined above and $\gamma_{n+1}=2^*\left(\frac{\gamma_n+2-2^*}{2}\right)$; and iterating as above, we get $u\in L^q(\mathbb{R}^N)$ for all $q\in [2,\gamma_n)$ and $n\in \mathbb{N}$. Since $\gamma_n \rightarrow\infty$, we get
    \begin{equation}\label{eq_4.15}
    	u\in L^q(\mathbb{R}^N) \text{ for all } q\geq 2.
    \end{equation}
    Now for the Sobolev regularity, following the proof of \cite[{Theorem 1.2}]{Giacomoni2025Normalized} we reformulate our problem as:
    $$-\Delta u +\mathcal{L}_1(u)+u = (I_{\alpha}*F(u))f(u)-\mathcal{L}_2(u),$$
    where
    $$\mathcal{L}_1(u):=C(N,s)P.V\int_{|x-y|\leq B}\frac{u(x)-u(y)}{|x-y|^{N+2s}}dy;$$
    and $$\mathcal{L}_2(u):=C(N,s)P.V\int_{|x-y|>B}\frac{u(x)-u(y)}{|x-y|^{N+2s}}dy,$$
    for some fixed $B>0$. Defining $g:=(I_{\alpha}*F(u))f(u)-\mathcal{L}_2(u)$, since 
    $$\int_{\rnn}|(I_{\alpha}*F(u))f(u)|^q\,dx\leq M_1C\int_{\rnn}\left(|u|^{\frac{q\alpha}{N}}+|u|^{\frac{q(\alpha+2)}{N-2}}\right)\,dx<+\infty \text{ for all } q\geq \frac{2N}{\alpha}, $$
    and, by Jensen's inequality,
    \begin{eqnarray*}
    	\int_{\rnn}|\mathcal{L}_2(u)|^q\,dx & \leq & C(N)\int_{\rnn}\left(\int_{|x-y|>B}\frac{|u(x)-u(y)|^q}{|x-y|^{q(N+2s)}}dy\right)dx\\
    	& \leq & 2C(N,q)\int_{\rnn}|u(x)|^q\left(\int_{|x-y|>B}\frac{1}{|x-y|^{q(N+2s)}}dy\right)dx\\
    	& \leq & C(B,N,q)\int_{\rnn}|u|^q\,dx<+\infty \text{ for all } q\geq \frac{2N}{\alpha},
    \end{eqnarray*}
     thus by \cite[Theorem~3.1.20]{Garroni2002Second} with $\Omega_I$ being the ball of radius $B$ centered at origin and a fixed bounded domain $\Omega$, get $u\in W^{2,q}_{loc}(\mathbb{R}^N)$ for all $q\geq \frac{2N}{\alpha}$ and hence in $W^{2,q}_{loc}(\mathbb{R}^N)$ for all $q\geq2$. Further using Sobolev inequality, we conclude that $u \in C^{1,\delta}_{\text{loc}}(\mathbb R^N)$, for all $\delta \in (0,1)$.
\end{myproof}\\
Now we are in a position to give a proof of Poho\v{z}aev identity in the spirit of \cite[Theorem 2.5]{Anthal2025Pohozaev}.
\begin{myproof}{Theorem}{\ref{Pohozaev_identity}}
Let $\psi \in C_c^1(\mb R^N)$ be such that $0 \leq \psi \leq 1$ in $\mb R^N$, $\psi\equiv 1$ for $|x| \leq 1$ and $\psi(x)\equiv 0$ for $|x| \geq 2$. We define $ \psi_\la(x) =\psi(\la x)$ for all $x \in \mb R^N$ and $\la >0$. Note that for all $x \in \mb R^N$ and $\la >0$, we get
		\begin{align}\label{e2.2}
			0 \leq \psi_\la (x) \leq 1~\text{and}~|x||\na \psi_\la(x)| \leq C,
		\end{align}
		where the constant $C>0$ is independent of $\la$. Also, we define the $C^1$ vector field as 
        \begin{equation}\label{psil}
        \Psi_\la(x)=\psi_\la(x)x.
        \end{equation}
        Now using \eqref{e2.2}, we conclude that
			\begin{align*}
				|\Psi_\la(x)-\Psi_\la(y)| \leq C|x-y|,
			\end{align*}
			for some constant $C>0$ independent of $\la$ and so $\Psi_\la \in C_c^{0,1}(\mb R^N, \mb R^N)$. Then choosing the test function $\varphi=\psi_\la\sum\limits_{j=1}^{N}x_j D_j u$ with $D_i$, the difference quotient operator defined as follows:
		$$D_iv(x)= \frac{v(x+he_i)-v(x)}{h},$$ 
        with $e_i=(0,\cdots , 1, \cdots,0)$ being the unit vector along   $x_{i}$ coordinate, 
        in \eqref{weak_sol}, we obtain
\begin{equation}\label{e4.8}
\begin{split}
   &\I{\mb R^N} \nabla u \cdot \nabla \varphi\, dx+  \I{\mb R^N}\I{\mb R^N} {(u(x)-u(y)) \varphi(x)-\varphi(y))}\,d\mu
   = \I{\mb R^N} ((I_\alpha \ast F(u)) f(u)-u )\varphi\, dx,
   \end{split}
\end{equation}
where for the sake of simplicity, we set $d\mu = \frac{dxdy}{|x-y|^{N+2s}}$ and drop the constant $C(N,s)$.
We observe that 
            \begin{equation}\label{MeI}
            \begin{split}
         \I{\mb R^N}\nabla u \cdot \nabla \varphi \,dx&= \I{\mb R^N}  \nabla u\cdot \nabla \left( \sum\limits_{j=1}^n \psi_\la x_j D_ju\right)\,dx \\
        &:= I_1 + I_2 + I_3,
        \end{split}
            \end{equation}
            where
            $$
            I_1=\sum\limits_{i,j=1}^N \I{\mb R^N}  \frac{\partial u }{\partial x_i}\frac{\partial \psi_\la }{\partial x_i}x_j D_j u\,dx,
            $$
            $$
            I_2=\sum\limits_{i,j=1}^N\I{\mb R^N} \frac{\partial  u}{\partial x_i} \psi_\la D_j(u) \delta_{ij} \,dx,
            $$
            and
            $$
            I_3=\sum\limits_{i,j=1}^N \I{\mb R^N}  \frac{\partial u}{\partial x_i}  \frac{\partial (D_j u)}{\partial x_i} {\psi_\la} x_j\,dx:=I_{3,1}-I_{3,2},
            $$
            where
            $$
    I_{3,1}=\frac{1}{2} \sum\limits_{j=1}^N \I{\mb R^N} D_j(|\nabla u|^2) \psi_\lambda x_j \,dx,
    $$
    and
    $$
    I_{3,2}=\frac{1}{2}\sum\limits_{j =1}^N \I{\mb R^N}(D_j(|\nabla u|^2) -\nabla u\cdot D_j(\nabla u))\psi_\lambda x_j \,dx.
    $$
    Here $\delta_{ij}$ being the Kronecker symbol, that is, $\delta_{ij}=1$ when $i=j$ and $0$ otherwise. Also,
\begin{equation}\label{nonl-II}
\begin{split}
&\I{\mb R^N}\I{\mb R^N}{(u(x)-u(y))(\varphi(x)-\varphi(y))}\,d\mu\\
=&\quad\sum\limits_{j=1}^N \I{\mb R^N}\I{\mb R^N}{(u(x)-u(y))(\psi_\lambda(x) x_j D_j u(x)-\psi_\lambda(y) y_jD_j u(y))}\,d\mu\\
:=&\quad J_{1} + J_{2}, 
\end{split}
\end{equation}
where
\begin{align*}
    J_{1} =\beta\sum\limits_{j=1}^N \I{\mb R^N}\I{\mb R^N}{(u(x)-u(y))(\psi_\lambda(x) x_j -\psi_\lambda(y) y_j)D_j u(x)}\,d\mu,
\end{align*}
and 
\begin{align*}
    J_{2} =&\beta\sum\limits_{j=1}^N \I{\mb R^N}\I{\mb R^N} {(u(x)-u(y))(D_j u(x) - D_j u(y))\psi_\lambda(y)y_j}\,d\mu\\
    =& \beta\sum\limits_{j=1}^N \I{\mb R^N}\I{\mb R^N}{(u(x)-u(y))(D_j u(x) - D_j u(y))\psi_\lambda(x)x_j}\,d\mu\\
    =& J_{2,1} - J_{2,2},
    \end{align*}
    where
    \begin{align*}
      J_{2,1} =\frac{\beta}{p} \sum\limits_{j=1}^N \I{\mb R^N}\I{\mb R^N}{D_j(|u(x)-u(y)|^2)\psi_\lambda(x) x_j}\,d\mu,
    \end{align*}
    and
\begin{align*}
    J_{2,2}
    =\frac{1}{2}\sum\limits_{j=1}^N\I{\mb R^N}\I{\mb R^N}\Big(D_j(|u(x)-u(y)|^2)-2(u(x)-u(y))(D_j u(x) - D_j u(y))\Big)\psi_\lambda(x)x_j\,d\mu.
\end{align*}
Therefore, taking into account \eqref{MeI} and \eqref{nonl-II}, equation \eqref{e4.8} reduces into
\begin{equation}\label{lnest-II}
I_1+I_2+I_{3,1}-I_{3,2}+(J_{1}+J_{2,1}-J_{2,2})=L,
\end{equation}
where
\begin{equation*}
L=\I{\mb R^N} ((I_\alpha \ast F(u)) f(u)-u ) \Big(\sum\limits_{j=1}^N \psi_\lambda x_j D_j u\Big)\, dx.
\end{equation*}
Now the estimates for $I_1$, $I_2$, $I_{3,1}$, $J_1$ and $J_{2,1}$ as $h \ra 0$ follow exactly on the lines of \cite[Theorem 2.5]{Anthal2025Pohozaev} and we obtain 
   \begin{equation}\label{estI1-I}
       \lim_{h\to 0} I_1=  \I{\mb R^N} (\nabla u\cdot \nabla \psi_\lambda)( x \cdot \nabla u)\,dx.
    \end{equation}
   \begin{equation}\label{estI2-I}
      \lim_{h\to 0} I_2 = 
        \I{\mb R^N} |\nabla u|^2\,\psi_\lambda \, dx,
    \end{equation}
 \begin{equation}\label{I31est-I}
    \begin{split}
      \lim_{h\to 0} I_{3,1}=-\frac{N}{2} \I{\mb R^N} |\nabla u|^2 \psi_\lambda \,dx -\frac{1}{2} \I{\mb R^N} |\nabla u|^2 \nabla \psi_\lambda \cdot x \,dx,
  \end{split}
   \end{equation}
   \begin{equation}\label{estJ1-II}
\begin{split}
   \lim_{h\to 0} J_1=&-\frac{1}{2}  \I{\mb R^N} \I{\mb R^N }\frac{|u(x)-u(y)|^2}{|x-y|^{N+2s}}\text{div}\Psi_\lambda(x)\,dxdy\\
&+\frac{(N+2s)}{2}\I{\mb R^N}\I{\mb R^N} |u(x)-u(y)|^2 \frac{(x-y)\cdot (\Psi_\lambda(x) -\Psi_\lambda(y))}{|x-y|^{N+2s+2}}\,dxdy,
\end{split}
\end{equation}
and
\begin{equation}\label{estJ21-II}
\begin{split}
    \lim_{h\to 0}J_{2,1} 
    &=-\frac{N}{2} \I{\mb R^N}\I{\mb R^N}{|u(x)-u(y)|^2\psi_\lambda(x)}\,d\mu-\frac{1}{2}\I{\mb R^N}\I{\mb R^N}{|u(x)-u(y)|^2 \nabla \psi_\lambda(x) \cdot x}\,d\mu.
    \end{split}
    \end{equation}
    One can notice that in order to estimate $J_1$ in \cite{Anthal2025Pohozaev}, it is assumed that the solution is in $C^{0,\ell}(\mathbb{R}^N)$ for some $\ell>s$. However, the presence of the function with compact support, $\psi$, permits us to generalize everything for $u\in C^{0,\ell}_{loc}(\mathbb{R}^N)$ (and then for $u\in C^{1,\delta}_{loc}(\mathbb{R}^N)$) as well (see also arguments in $B_R(0)\times B_R(0)$ with fixed suitable $R>0$ after (3.3) page 2010  in \cite{ambrosio}).\\
\textbf{Estimate of $I_{3,2}+J_{2,2}$:} Using the convexity of the map $t \mapsto |t|^2$,  we deduce that
    \begin{align}\label{e4.3}
    D_j(|\nabla u|^2)& -2 \nabla u\cdot D_j(\nabla u)
   \geq 0,
    \end{align}
    and 
\begin{align}\label{e4.9}
    D_j(|u(x)-u(y)|)^2-2(u(x)-u(y))(D_j u(x)-D_j u(y)) \geq 0,
\end{align}
for every $x,y\in\mb R^N$.  We observe that for any {$K\in L^1(\mathbb{R}^N)$}, we have
    \begin{align}\label{e4.4}
       \I{\mb R^N} D_j K \,dx = \frac1h\I{\mb R^N} (K(x+he_j) -K(x))\,dx
       =\frac1h \left( \I{\mb R^N} K(x) \,dx - \I{\mb R^N} K(x) \,dx \right)
       =0.
    \end{align}
Using \eqref{e4.4}, we get
    \begin{align}\label{e4.5}
        \I{\mb R^N} D_j|\na u|^2 \,dx = \I{\mb R^N}D_j((I_\alpha \ast F(u))F(u))\,dx=0.    \end{align}
Further, using similar arguments as in \eqref{e4.4} for double integral, we have
\begin{align}\label{e4.10}
    \I{\mb R^N}\I{\mb R^N} {D_j(|u(x)-u(y)|^2)}\,d\mu =0,
\end{align}
\noindent Choosing $\phi=\sum_{j=1}^{N}D_j u$ as a test function in \eqref{weak_sol}, using \eqref{eq I_alpha}, \ref{f4} and combining the estimates \eqref{e4.3}, \eqref{e4.9}, \eqref{e4.5} and \eqref{e4.10}, we have
\begin{align*}
    \Big|&I_{3,2} + J_{2,2}\Big|\\
    \leq& C \sum\limits_{j=1}^N\left( \I{\mb R^N}{\Big(}D_j(|\nabla u|^2 )-2 \nabla u \cdot \nabla( D_ju ){\Big)} dx \right.\\
    &+ \left. \I{\mb R^N}\I{\mb R^N}{\Big(}{D_j(|u(x)-u(y)|^2)-2(u(x)-u(y))(D_j u(x) - D_j u(y))}{\Big)}\,d\mu\right)\\
    =&2C \sum\limits_{j=1}^N \I{\mb R^N} (D_j(I_\alpha \ast F(u)) F(u) - (I_\alpha\ast F(u))f(u) D_j u)\,dx\\
    \leq& 2\|I_\alpha \ast F(u)\|_\infty C \sum\limits_{j=1}^N\I{\mb R^N}(D_jF(u)-f(u)D_j u) \,dx,\\
 =& 2C \sum\limits_{j=1}^N\I{\mb R^N}(f(\tilde{u})-f(u))D_j u \,dx,
    \end{align*}
    where $C = C(\lambda,u,\alpha)>0$ is a constant, $\tilde{u}$ is between $u$ and $u(\cdot + he_j)$. Now using the H\"{o}lder's inequality in the above estimate, we obtain
    \begin{equation}\label{I32--I}
        \Big|I_{3,2}+J_{2,2}\Big| \leq C \sum\limits_{j =1}^N \| f(\tilde{u})-f(u)\|_2\|D_j u\|_2\leq  C \sum\limits_{j =1}^N \| f(\tilde{u})-f(u)\|_2\| \nabla u\|_2.
    \end{equation}
     Since $u( \cdot + he_j) \rightarrow u$ in $L^2(\mb R^N)$, we obtain $\tilde{u} \rightarrow u$ in $L^2(\mb R^N)$. Finally using \ref{f4}, we have $|f(\tilde{u})| \leq |\tilde{u}|$ and since $|\tilde{u}| \ra |u|$ in $L^{2}(\mathbb{R}^N)$, by the generalized Lebesgue's dominated convergence theorem, we have
    \begin{equation}\label{fu}
        \lim_{h\to 0}\|f(\tilde{u})-f(u)\|_2=0.
    \end{equation}
Using \eqref{fu} and that $\|\nabla u\|_2 < \infty$ in \eqref{I32--I}, we obtain 
\begin{equation}\label{I32J22-II}
\lim_{h\to 0}(I_{3,2}+J_{2,2})=0.
\end{equation}
    \textbf{Estimate of $L$:}  We have
    \begin{equation}\label{RHS-I}
    \begin{split}
       \lim_{h\to 0}L&=\lim_{h\to 0}\I{\mb R^N} ((I_\alpha \ast F(u))f(u)-u) \Big(\sum\limits_{j=1}^N \psi_\lambda(x) x_j D_j u\Big)\, dx\\
        & =  \I{\mb R^N}((I_\alpha \ast F(u))f(u)-u) \psi_\lambda(x) \nabla u\cdot x \,dx.
        \end{split}
    \end{equation}
Combining the estimates \eqref{estI1-I}, \eqref{estI2-I}, \eqref{I31est-I}, \eqref{I32J22-II}, \eqref{estJ1-II}, \eqref{estJ21-II} and \eqref{RHS-I} in \eqref{lnest-II}, we obtain
\begin{align}\nonumber \label{e4.11}
  &\I{\mb R^N}((I_\alpha \ast F(u))f(u)-u) \psi_\lambda(x) \nabla u\cdot x \,dx\\ \nonumber
  =  & \I{\mb R^N} (\nabla u \cdot \nabla \psi_\lambda(x))( x \cdot \nabla u)\,dx+ \frac{2-N}{2} \I{\mb R^N} |\nabla u|^2 \psi_\lambda(x) \,dx \nonumber
   -\frac{1}{2} \I{\mb R^N} |\nabla u|^2 \nabla \psi_\lambda(x) \cdot x \,dx\\ \nonumber
   &-\frac{N}{2} \I{\mb R^N}\I{\mb R^N}{|u(x)-u(y)|^2\psi_\lambda(x)}\,d\mu-\frac{1}{2}\I{\mb R^N}\I{\mb R^N}{|u(x)-u(y)|^2 \nabla \psi_\lambda(x) \cdot x}\,d\mu\\ \nonumber
   &-\frac{1}{2}  \I{\mb R^N} \I{\mb R^N }{|u(x)-u(y)|^2}\text{div}\Psi_\lambda(x)\,d\mu\\
&+\frac{N+2s}{2}\I{\mb R^N}\I{\mb R^N} |u(x)-u(y)|^2 \frac{(x-y)\cdot (\Psi_\lambda(x) -\Psi_\lambda(y))}{|x-y|^{N+2s+2}}\,dxdy.
\end{align}
           Finally, we pass on to the limits $\lambda \rightarrow 0$ in \eqref{e4.11}. To this end, using \eqref{e2.2}, we notice that, for every $x,y \in \mb R^N$ with $x \ne y$ and $\la >0$,
	\begin{align}\label{e2.7}
		\left| \text{div}(\Psi_\la(x))+\text{div}(\Psi_\la(y))-(N+2s)\frac{(\Psi_\la(x)-\Psi_\la(y))\cdot(x-y)}{|x-y|^2}\right| \leq C,
	\end{align}
	for some constant $C >0$ independent of $\la$.
	Using \eqref{e2.2}, \eqref{e2.7} along with the pointwise convergences $\psi_\la(x) \ra 1$, $\na \psi_\la(x)\cdot x \ra 0$, $\Psi_\la (x) \ra x$ and $\text{div}\, \Psi_\la(x) \ra N$ for all $x \in \mb R^N$ as $\la \ra 0$ and arguing as in \cite[Proposition~3.1]{Moroz2013groundstates}, we apply the Lebesgue's dominated convergence theorem to obtain
	\begin{align*}
	    \left(\frac{N-2}{2}\right) \| \nabla u\|_2^2 + \left(\frac{N-2s}{2}\right) [u]^2 + \frac{N}{2} \|u\|_2^2 = \left(\frac{N+\alpha}{2}\right)A(u),
	\end{align*}
    which is the required identity. 
\end{myproof}

\section*{Acknowledgments}

Prashanta Garain thanks IISER Berhampur for the seed grant:\\ IISERBPR/RD/OO/2024/15, Date: February
08, 2024. 

\noindent Nidhi Nidhi (PMRF ID - 1402685), is supported by the Ministry of Education, Government of India, under the Prime Minister’s Research Fellows (PMRF) scheme.

\noindent Gurdev Chand Anthal is currently working as a PDRF under the supervision of Prof. Minbo Yang. He acknowledges the support and guidance of Prof. Minbo Yang.

\bibliographystyle{abbrv}   
\bibliography{main}

 \noindent \textsf{Gurdev Chand Anthal }

 \noindent \textsf{School of Mathematical Sciences}
 
 \noindent \textsf{Zhejiang Normal University}
 
\noindent \textsf{Jinhua 321004, People's Republic of China.}
 
\noindent\textsf{e-mail}: gurdevanthal92@gmail.com
\vspace{0.75cm}

\noindent \textsf{Prashanta Garain}

\noindent \textsf{Department of Mathematical Sciences}

\noindent \textsf{Indian Institute of Science Education and Research Berhampur}

\noindent \textsf{Permanent Campus, At/Po:-Laudigam, Dist.-Ganjam}

\noindent \textsf{Odisha, India-760003}
\vspace{0.75cm}

\noindent \textsf{Nidhi Nidhi}

\noindent \textsf{Department of Mathematics}

\noindent \textsf{Indian Institute of Technology Delhi}

\noindent \textsf{Hauz Khas, New Delhi-110016, 
India}

\noindent \textsf{e-mail: nidhi.kaushik2809@gmail.com}

\end{document}